\documentclass[11pt]{amsart}
\usepackage{cgeometry}
\usepackage{bigints}

\bibliography{ComplexGeometry}
\begin{document}

\title{Pluripotential-theoretic stability thresholds}
\author{Mingchen Xia}
\date{\today}
\maketitle

\begin{abstract}
Given a compact polarized manifold $(X,L)$, we introduce two new stability thresholds in terms of singularity types of global quasi-plurisubharmonic functions on $X$. We prove that in the Fano setting, the new invariants can effectively detect  K-stability of $X$. We study some functionals of geodesic rays in the space of K\"ahler potentials by means of the corresponding test curves. In particular, we introduce a new entropy functional of quasi-plurisubharmonic functions and relate the radial entropy functional to this new entropy functional.
\end{abstract}


\tableofcontents


\section{Introduction}

Let $X$ be a complex projective manifold of dimension $n$. Let $L$ be an ample line bundle on $X$. Fix a Kähler form $\omega\in c_1(L)$. Let $V=(L^n)$.

A central problem in K\"ahler geometry is to give conditions for the existence of canonical metrics in the K\"ahler class $[\omega]$. The celebrated Yau--Tian--Donaldson conjecture asserts that the existence of K\"ahler--Einstein metrics (or more generally cscK metrics) is equivalent to certain algebro-geometric stability conditions. Classically, dating back to the work of Ding--Tian (\cite{DT92}), Tian (\cite{Tia97}) and Donaldson (\cite{Don02}), the algebro-geometric stability notion, known as \emph{K-stability} has been defined in terms of test configurations. Later on, a stronger condition known as \emph{uniform K-stability} is also introduced and studied in \cite{Der16} and in \cite{BHJ16}, \cite{BHJ17}. It is known that the equivariant version of uniform K-stability gives a characterization of the existence of K\"ahler--Einstein metrics, which even generalizes to the log Fano setting, see \cite{Li22} and references therein.

On the other hand, more recently a valuative approach to K-stability is introduced in \cite{Fuj19}, \cite{FO18}, \cite{BJ20}, which we briefly recall. The $\delta$-invariant of $L$ is defined as
\begin{equation}\label{eq:delta}
\delta(L):=\inf_E\frac{A_X(E)}{S_L(E)}\,,
\end{equation}
where $E$ runs over the set of prime divisors over $X$, $A_X(E)$ denotes the log discrepancy of $E$ and $S_L(E)$ is the expected order of vanishing of $L$ along $E$. It is well-known that uniform twisted K-stability (resp. twisted K-semistability) is equivalent to $\delta(L)>1$ (resp. $\delta(L)\geq 1$). See \cite{Fuj19}, \cite{Li17}, \cite{FO18}, \cite{BJ18} for details.

In this paper, we introduce a different stability threshold in terms of the singularity types of quasi-plurisubharmonic functions on $X$: 
\begin{equation}\label{eq:deltanew}
\delta_{\mathrm{pp}}:=\inf_{[\psi]}\frac{\int_{-\infty}^{\infty}\Ent([\psi^+_{\tau}])\,\mathrm{d}\tau}{nV^{-1}\int_{-\infty}^{\infty} \left(\int_X \omega\wedge \omega_{\psi^+_{\tau}}^{n-1}-\int_X\omega_{\psi^+_{\tau}}^n\right)\,\mathrm{d}\tau}\,,
\end{equation}
  where $[\psi]$ runs over the set of singularity types of quasi-psh functions with some non-zero Lelong numbers on $X$, $\psi^+_{\bullet}$ is a test curve associated to $\psi$.
   Recall that a test curve is the Legendre transform of a geodesic ray as in \cite{RWN14}. The test curve $\psi^+_{\bullet}$ is the maximal extension of the test curve corresponding to deformation to the normal cone (see \cref{subsec:extdef}). The quantity $\Ent[\bullet]$ is an invariant of the singularity types of quasi-psh functions (see \cref{def:naentgeneral}). To the best of the author's knowledge, this invariant has never been studied in the literature. Observe that the quotient in \eqref{eq:deltanew} depends only on the singularity type of $\psi$. Moreover, it is easy to see that the quotient in \eqref{eq:deltanew} does not change under the rescaling $\psi\to c\psi$ for $c\in \mathbb{R}_{>0}$, hence one could restrict $\psi$ to run only in the set of $\omega$-psh functions.

 Now we state the main theorem. 
\begin{theorem}\label{thm:introdelta1}
Let $(X,L)$ be a polarized manifold. Then $\delta_{\mathrm{pp}}\geq \delta$. Further, if $X$ is Fano and $L=-K_X$ and $\delta<\frac{n+1}{n}$, then $\delta=\delta_{\mathrm{pp}}$.
\end{theorem}

We recall that in the Fano setting, when $\delta\leq 1$, $\delta$ is also equal to the greatest Ricci lower bound, see \cref{subsec:min} for details. We remark that although our theorem concerns only K\"ahler geometry, our proof relies essentially on the non-Archimedean tools developed by Boucksom--Jonsson, as we recall later.

As a corollary,
\begin{corollary}\label{cor:introdelta1}
Assume that $X$ is Fano and $L=-K_X$. Then 
\begin{enumerate}
    \item $\delta_{\mathrm{pp}}\geq 1$ if{f} $X$ is K-semistable.
    \item $\delta_{\mathrm{pp}}> 1$ if{f} $X$ is uniformly K-stable.
\end{enumerate}
\end{corollary}
\cref{cor:introdelta1} integrates into the program of characterizing K-stability in terms of some more explicit data dating back to  \cite{RT07}. 
In \cite{RT07}, Ross--Thomas introduced the notion of slope stability in terms of test configurations associated to deformation to the normal cone, which gives a necessary condition for K-stability. 
Later on, this theory was extended in \cite{Oda13} using flag ideals, in \cite{WN12}, \cite{Sze15} using filtrations.
In \cite{Fuj19} and \cite{Li17}, Fujita and Li give a characterization of K-stability in terms of all divisorial valuations.
Our result gives a different characterization in terms of psh singularity types. Our approach can also be seen as a generalization of those of \cite{RT07} in the sense that our definition of $\delta_{\mathrm{pp}}$ is based on a generalization of deformation to the normal cone.
We also notice that very recently in \cite{DL22}, Dervan--Legendre have partially extended Fujita's work to general polarizations and studied valuative stability. 

When the stability threshold is less than $1$, it is interesting to understand its minimizers. 
We propose the following conjecture:
\begin{conjecture}When $\delta_{\mathrm{pp}}\leq 1$, there is always a minimizer of $\delta_{\mathrm{pp}}$. 

When $X$ is a Fano manifold, $L=-K_X$ and when $\delta<1$, the pluricomplex Green function $G$ in the sense of \cite{McCT19} is a minimizer of $\delta_{\mathrm{pp}}$. 
\end{conjecture}
The first part in the Fano case follows from our proof of the main theorem. For general polarization, it seems difficult.

There are some similar results for $\delta$ in the more general log Fano variety setting: there is always a quasi-monomial valuation that computes $\delta$ (\cite{BLX19}). In the smooth Fano setting, there is a divisor computing $\delta$ (\cite{DS20}, see \cite[Theorem~6.7]{BLZ19} for details). The same holds in the log Fano setting by the  recent breakthrough \cite{LXZ21}.

The $\delta_{\mathrm{pp}}$-invariant is closely related to $\delta$ in the following manner: take an extractable (\cref{def:ext}) divisor $E$. 
One can prove that in this case, there is always an $\omega$-psh function $\psi$ with analytic singularities such that on a suitable birational model $\pi:Y\rightarrow X$, the singularities of $\pi^*\psi$ are just the hyperplane singularity along $E$.
Recall that $E$ induces a test configuration $(\mathcal{X},\mathcal{L})$. 
Now one can make explicit computations to express various functionals of $(\mathcal{X},\mathcal{L})$ in terms of $\psi_{\bullet}$, the result turns out to be of the form of \eqref{eq:deltanew}. More precisely, we prove that for a test curve induced by a general (semi-ample) test configuration, the non-Archimedean entropy and $\tilde{J}^{\NA}$-functionals (the latter is more frequently denoted by $I^{\NA}-J^{\NA}$ in the literature) are both integrals of some corresponding functionals of psh singularities along the test curve (see \cref{thm:Entcomp2}, \cref{cor:eqholds}, \cref{cor:Jtildeslope}). Conversely, given an $\omega$-psh function $\psi$ with analytic singularities, we can always take a log  resolution so that $\psi$ has singularities along a snc (strictly normal crossing) $\mathbb{Q}$-divisor $D=\sum_{i}a_i D_i$. This divisor then induces a higher rank valuation $(a_i^{-1} \ord_{D_i})_i$ of $\mathbb{C}(X)$. 

As a byproduct of our work, we could also define a slightly different stability threshold:
\begin{equation}\label{eq:t9}
\delta':=\inf_{\psi}\frac{(K_{Y/X}\cdot (-\Div_Y\psi)^{n-1})+n\left(G_{n-1}(L,\Div_Y\psi)\cdot \Redu \Div_Y\psi\right)}{n\int_0^1 \left(\int_X \omega\wedge \omega_{\tau\psi}^{n-1}-\int_X  \omega_{\tau\psi}^{n}\right)\,\mathrm{d}\tau}\,,
\end{equation} 
where $[\psi]$ runs over the set of singularity types of unbounded $\omega$-psh functions with analytic singularities,
$\pi:Y\rightarrow X$ is a log resolution of $\psi$, $G_{n-1}$ is a polynomial defined by \eqref{eq:Gn}, $\Redu$ of a divisor $D$ is the divisor with the same support as $D$ but with all non-zero coefficients of $D$ set to $1$. Note that the quotient in \eqref{eq:t9} is \emph{not} invariant under the rescaling $\psi\mapsto c\psi$ ($c\in \mathbb{Q}_{>0}$), hence $\delta'$ is an invariant of $\omega$-psh functions on $X$, not an invariant of all quasi-psh functions on $X$ as $\delta_{\mathrm{pp}}$ is.
We also prove that
\begin{theorem}\label{thm:introdelta2}
We always have $\delta'\geq \delta$.
\end{theorem}
In general, we do not expect $\delta$ and $\delta'$ to be equal even if $\delta\leq 1$. The invariant $\delta'$ restricts the possible singularities of an $\omega$-psh function. Although $\delta'$ does not seem to have direct applications in K-stability, it may play some interesting roles in pluripotential theory.

\textbf{Philosophy behind the theorems}

Before discussing the proofs, let us explain the philosophy behind these theorems. 

We regard the global pluripotential theory of singular metrics on a compact Kähler manifold as a differential version of the theory of geodesic rays in the space of Kähler potentials. 
 In fancier terms, we could roughly regard the space of quasi-psh singularity types as the boundary at infinity of the space of geodesic rays: on one hand, each  quasi-psh singularity type induces a geodesic ray; on the other hand, each geodesic ray degenerates to a quasi-psh singularity type at infinity.
 As in the finite dimensional picture between the space of rays in $\mathbb{R}^n$ and the sphere at infinity $S^{n-1}$, it is natural to expect a closer relation between these spaces.
One of the justifications is given by \cref{thm:LegLk} (namely, \cite[Theorem~1.1]{DX22}). Similarly, by the computations in this paper and in \cite{DX22}, a number of radial invariants of geodesic rays are in fact an integral along the corresponding test curves of some corresponding invariants defined by quasi-psh functions. See \cref{tbl:Comp} for more examples.

Classically, K-stability of a polarized manifold is detected by test configurations, valuations, filtrations and non-Archimedean potentials, which can all be embedded in the space of geodesic rays. By our philosophy, there should be a pluripotential-theoretic counterpart, which leads to the present paper. 
Note that our previous work \cite{DX22} already followed this philosophy.

We do not expect \cref{thm:introdelta1}, \cref{thm:introdelta2} to be useful when trying to find new examples of K-stable varieties. However, from the pluripotential-theoretic point of view, these results provide strong restrictions on the possible singularity types of quasi-psh functions using global geometric conditions. To the best of my knowledge, this kind of results has never been studied before.

\textbf{Strategy of the proof}

In the discussion, we fix a maximal geodesic ray in the sense of \cite{BBJ15} and its Legendre transform $\psi=\hat{\ell}$.

As discussed above, we need to express various energy functionals of $\ell$ in terms of $\psi_{\bullet}$.

The part for $\tilde{\mathbf{J}}$ follows from the strategy introduced in \cite{RWN14} and further developed in \cite{DX22}. See the proof of \cref{thm:LegEalpha} for details.

The corresponding result for the entropy functional is the main new feature in this paper. As the variation of the entropy functional is not easily controlled, we try to tackle the non-Archimedean counterpart of the entropy at first, namely the non-Archimedean entropy:
\[
\Ent^{\NA}(\phi):=\frac{1}{V}\int_{X^{\An}} A_X\,\MA(\phi)\,,\quad \phi\in \mathcal{E}^{1,\NA}\,.
\]
In \cite{Li20}, Li showed that $\Ent^{\NA}(\phi)$ is dominated by the slope at infinity of the usual entropy functional $\Ent$ along $\ell$, where $\phi$ is the non-Archimedean potential induced by $\ell$.
In \cite{DX22}, we have expressed the non-Archimedean Monge--Amp\`ere energy in terms of the test curves. Since the non-Archimedean Monge--Amp\`ere energy is nothing but the primitive function of the Chambert-Loir measure $\MA(\phi)$, we get \emph{a fortiori} a good understanding of $\MA(\phi)$. We make use of this description to compute the non-Archimedean entropy functional. The result turns out to be an integral of the variation of volumes along $\psi_{\bullet}$.

Recall the potentials $\psi_{\tau}$ are all $\mathscr{I}$-model in the sense of \cite{DX22}.
In order to compute the variation of volumes of an $\mathscr{I}$-model potential, we need to express this volume algebraically. We prove that the volume of an $\mathscr{I}$-model potential can be realized as certain (movable) intersection number of a b-divisor (in the sense of Shokurov) associated to the singularities of the potential, if the intersection number is properly defined (see \cref{thm:volIm}). Now we can carry out a purely algebraic computation to get a formula for the non-Archimedean entropy (See \cref{thm:Entcomp2}, \cref{cor:eqholds}).

Now it comes to \cref{thm:introdelta1} and \cref{thm:introdelta2}. That these new invariants dominate $\delta$ is an easy consequence of the formulae of $\tilde{\mathbf{J}}$ and $\Ent^{\NA}$. We simply embed the set of $\omega$-psh functions into the set of maximal geodesic rays using the deformation to the normal cone like construction. This kind of embedding was already studied in \cite{DDNLmetric}.
For the equality $\delta=\delta_{\mathrm{pp}}$ when $\delta\leq 1$. We make use of the results from \cite{BLZ19}, which says that when $\delta\leq 1$, $\delta$ can be computed by a divisor $E$. Then \cite{BCHM10} allows us to extract the divisor. We can therefore construct an $\omega$-psh $\psi$ whose singularities are exactly given by $E$. Then $\psi$ minimizes $\delta_{\mathrm{pp}}$ as well and we conclude that $\delta=\delta_{\mathrm{pp}}$.

We remark that although we have carried out our computations only on smooth complex varieties, it is easy to generalize most results to normal Kähler varieties. However we decide to limit ourselves to the smooth setting in order to keep the present paper at a readable length.

\textbf{Organization of the paper}

In \cref{sec:setup}, we present a few results necessary for understanding the definition of $\delta_{\mathrm{pp}}$.

In \cref{sec:pre} and \cref{sec:rwn}, we recall some basic notions in Kähler geometry and pluripotential theory.

In \cref{sec:bdiv}, we recall the notion of Shokurov's b-divisors and apply it to define the entropy of qpsh singularities.

In \cref{sec:for} and \cref{sec:vBer}, we express several functionals on the space of geodesic rays in terms of the corresponding test curves.


In \cref{sec:delta}, we relate the new $\delta$-invariants to the classical $\delta$-invariant.

In \cref{sec:pro}, we propose several further problems.

\textbf{Conventions}

In this paper, all Monge--Ampère type operators are taken in the non-pluripolar sense (see \cite{BEGZ10}). The functional $\tilde{J}$ defined in \eqref{eq:Jtil} is usually written as $I-J$ in the literature.
Our definition of \emph{test curves} in \cref{def:testcurve} corresponds to \emph{maximal test curves} in the literature (see \cite{RWN14} for example). The $\mathrm{d}^{\mathrm{c}}$ operator is normalized so that $\ddc=\frac{i}{2\pi}\partial \bp$. The definition of a birational model in \cref{def:BM} requires that the model be smooth, hence stronger than the usual definition. When $\omega$ is a Kähler form, we adopt the convention that $\omega_{-\infty}=\omega+\ddc(-\infty)=0$. A snc divisor is always assumed to be effective.
By a valuation of a field, we refer to real valuations unless otherwise specified. We adopt the additive convention for valuations. We allow ($\mathbb{Q}$-)Weil divisors to have countably many components.

We do not distinguish a holomorphic line bundle and the corresponding invertible sheaf in the \emph{analytic} category. We use interchangeably additive and multiplicative notations for tensor products of invertible sheaves.

We make use of the results of \cite{BHJ19} in an essential way. We only refer to the latest version on arXiv \cite{BHJ16} with errata instead of the journal version.

\textbf{Acknowledgements}

The author would like to thank Robert Berman, Chen Jiang, Kewei Zhang for discussions and Tamás Darvas, Sébastien Boucksom, Mattias Jonsson, Yuchen Liu, Ruadhaí Dervan, Yaxiong Liu for remarks and suggestions on early versions of this paper. 
The author also wants to thank the referees for many valuable suggestions.

Part of the paper is used as assignments of the course \emph{GFOK035 Academic Writing} at Chalmers Tekniska H\"ogskola. The author would like to thank his classmates Mohammad Farsi and Víctor López Juan for their (non-mathematical) suggestions.

\section{Setup}\label{sec:setup}
In this section, we present a minimal amount of preliminaries necessary to understand the definition of the new delta invariant \eqref{eq:deltanew}.

\subsection{Log resolution of analytic singularities}
Let $X$ be a projective manifold of dimension $n$. Let $L$ be a big and semi-ample line bundle on $X$. Let $h$ be a smooth non-negatively curved Hermitian metric on $L$. Let $\omega=c_1(L,h)$. Let $\PSH(X,\omega)$ denote the set of all $\omega$-psh functions on $X$, namely the set of usc functions $\varphi:X\rightarrow [-\infty,\infty)$ such that $\omega+\ddc \varphi\geq 0$ as currents. See \cite{GZ17} for more details.

\begin{definition}\label{def:ana}
 A potential $\varphi\in \PSH(X,\omega)$ is said to have \emph{analytic singularities} if for each $x\in X$, there is a neighbourhood $U_x\subseteq X$ of $x$ in the Euclidean topology, such that on $U_x$,
 \[
 \varphi=c\log\left(\sum_{j=1}^{N_x}|f_j|^2\right)+\psi\,,
 \]
 where $c\in \mathbb{Q}_{\geq 0}$, $f_j$ are analytic functions on $U_x$, $N_x\in \mathbb{Z}_{>0}$ is an integer depending on $x$, $\psi\in C^{\infty}(U_x)$.
\end{definition}

\begin{definition}\label{def:anaD}
Let $D$ be an effective snc $\mathbb{R}$-divisor on $X$.   Let $D=\sum_i a_i D_i$ with $D_i$ being prime divisors and $a_i\in \mathbb{R}_{>0}$. We say that 
 $\varphi\in \PSH(X,\omega)$ has \emph{analytic singularities along $D$} if locally (in the Euclidean topology), 
 \[
 \varphi=\sum_i a_i\log|s_i|_h^2+\psi\,,
 \]
 where $s_i$ is a local section of $L$ that defines $D_i$, $\psi$ is a smooth function.
\end{definition} 
 Note that a potential with analytic singularities along a snc $\mathbb{Q}$-divisor has analytic singularities in the sense of \cref{def:ana}. 
 
\begin{definition}\label{def:BM}
 A \emph{birational model} of $X$ is a projective birational morphism $\pi:Y\rightarrow X$ from a \emph{smooth} projective variety $Y$ to $X$.
\end{definition}
\begin{definition}
Let $\varphi\in \PSH(X,\omega)$ be a potential with analytic singularities. Then there is a birational model $\pi:Y\rightarrow X$ of $X$, such that $\pi^*\varphi$ has analytic singularities along a snc $\mathbb{Q}$-divisor (see \cite[Lemma~2.3.19]{MM07}).
We call any such $\pi$ a \emph{log resolution} of $\varphi$.
\end{definition}

\subsection{\texorpdfstring{$\mathscr{I}$}{I}-model potentials}\label{subsec:Imod}
Let $X$ be a compact Kähler manifold of dimension $n$. Let $L$ be an ample line bundle. Let $\omega\in c_1(L)$ be a K\"ahler form.   For any quasi-psh function $\varphi$ on $X$, let $\mathscr{I}(\varphi)$ denote Nadel's multiplier ideal sheaf of $\varphi$, namely, the coherent ideal sheaf on $X$ locally generated by holomorphic functions $f$ such that $\int |f|^2 e^{-\varphi}\omega^n<\infty$.

The concept of $\mathscr{I}$-model potential is developed in \cite{DX22}. 

\begin{definition}\label{def:qeqa}
 Let $\varphi\in \PSH(X,\omega)$. A \emph{quasi-equisingular approximation} of $\varphi$ is a sequence $\varphi^j$ of potentials in $\PSH(X,\omega)$ with analytic singularities, such that
 \begin{enumerate}
     \item $\varphi^j$ converges to $\varphi$ in $L^1$.
     \item The singularity types of $\varphi^j$ are decreasing.
     \item For any $\delta>0$, $k>0$, we can find $j_0=j_0(\delta,k)>0$, so that for $j\geq j_0$,
     \[
     \mathscr{I}((1+\delta)k\varphi^j)\subseteq \mathscr{I}(k\varphi)\,.
     \]
 \end{enumerate}
\end{definition}
Recall that quasi-equisingular approximations always exist (\cite[Lemma~3.2]{Cao14}, \cite[Theorem~2.2.1]{DPS01}).

Recall that a potential $\varphi\in \PSH(X,\omega)$ is said to be \emph{$\mathscr{I}$-model} if
\[
\varphi=P[\varphi]_{\mathscr{I}}:=\sups\left\{\,\psi\in \PSH(X,\omega): \psi\leq 0,\psi^{\An}\leq \mathscr{I}(k\psi)\subseteq \mathscr{I}(k\varphi) \text{ for all }k\in \mathbb{N} \,\right\}\,.
\]
We use the notation $\PSH^{\Mdl}_{\mathscr{I}}(X,\omega)$ to denote the set of $\mathscr{I}$-model potentials in $\PSH(X,\omega)$.

\begin{theorem}[{\cite[Theorem~1.4]{DX22}}]\label{thm:DX1-4}
	Let $\varphi\in \PSH^{\Mdl}(X,\omega)$, $\int_X \omega_{\varphi}^n>0$. Then the following are equivalent:
	\begin{enumerate}
		\item $\varphi\in \PSH^{\Mdl}_{\mathscr{I}}(X,\omega)$.
		\item 
		\[
			\varphi=\sups\left\{\,\psi\in \PSH(X,\omega): \psi\leq 0,\mathscr{I}(k\psi)\subseteq \mathscr{I}(k\varphi)\text{ for any }k\in \mathbb{R}_{>0} \,\right\}\,.
		\]
		\item 
		\[
		\lim_{k\to\infty}\frac{n!}{k^n}h^0(X,K_X\otimes L^k\otimes \mathscr{I}(k\varphi))=\int_X \omega_{\varphi}^n\,.
		\]
		\item For one (or equivalently any) quasi-equisingular approximation $\varphi^j$ of $\varphi$,
		\[
		\lim_{j\to\infty} \int_X \omega_{\varphi^j}^n=\int_X \omega_{\varphi}^n\,.
		\]
	\end{enumerate}
\end{theorem}
In terms of the function $\varphi^{\An}$ introduced below, these conditions are also equivalent to 
\[
	\varphi=\sups\left\{\,\psi\in \PSH(X,\omega): \psi\leq 0,\psi^{\An}\leq \varphi^{\An}\,\right\}\,.
\]

Here and in the whole paper, products like $\omega_{\varphi}^n$ are taken in the non-pluripolar sense, see \cite{BEGZ10}.

For the definition of model potentials, we refer to \cite{DDNL18mono}. The set of model potentials in $\PSH(X,\omega)$ is denoted  by $\PSH^{\Mdl}(X,\omega)$.
Recall that a model potential with analytic singularities is $\mathscr{I}$-model (\cite{Bon98}).

Let $X^{\Div}_{\mathbb{Q}}$ denote the set of all $\mathbb{Q}$-divisorial geometric valuation on $X$. Namely, elements of $X^{\Div}_{\mathbb{Q}}$ are $c\ord_E$, where $c\in \mathbb{Q}_{>0}$, $E$ is a prime divisor over $X$ (i.e. a prime divisor on a birational model of $X$).
Let $\psi\in \PSH(X,\omega)$, recall that $\psi^{\An}$ is a function on  $X^{\Div}_{\mathbb{Q}}$ defined as follows: let $v\in X^{\Div}_{\mathbb{Q}}$, then set
\begin{equation}\label{eq:psian}
-v(\psi)=\psi^{\An}(v):=-\lim_{k\to\infty} \frac{1}{k} v\left(\mathscr{I}(k\psi)\right)\,.
\end{equation}

\begin{lemma}\label{lma:qesana}
Let $\varphi\in \PSH(X,\omega)$. Let $\varphi^j$ be a quasi-equisingular approximation of $\varphi^j$. Then $\varphi^{j,\An}\to \varphi^{\An}$ pointwisely on $X^{\Div}_{\mathbb{Q}}$ as $j\to\infty$.
\end{lemma}
\begin{proof}
It suffices to prove that for any prime divisor $E$ over $X$, $\varphi^{j,\An}(\ord_E)\to \varphi^{\An}(\ord_E)$.
Fix $k\in \mathbb{Z}_{>0}$, $\delta\in \mathbb{Q}_{>0}$, take $j_0>0$, so that when $j>j_0$, $\mathscr{I}((1+\delta)k\varphi^j)\subseteq \mathscr{I}(k\varphi)$.
When $j>j_0$, we get
\[
\frac{1}{k}\ord_E(\mathscr{I}(k\varphi))\leq \frac{1}{k}\ord_E(\mathscr{I}((1+\delta)k\varphi^j))\,.
\]
By Fekete's lemma, 
\[
-\varphi^{j,\An}(\ord_E)=\sup_{k\in \mathbb{Z}_{>0}}\frac{1}{k}\ord_E(\mathscr{I}(k\varphi^j))\,.
\]
So
\[
\frac{1}{k}\ord_E(\mathscr{I}(k\varphi))\leq (1+\delta)(-\varphi^{j,\An}(\ord_E))\,.
\]
Take sup with respect to $k\in \mathbb{Z}_{>0}$, we get
\[
-\varphi^{\An}(\ord_E)\leq (1+\delta)(-\varphi^{j,\An}(\ord_E))\,.
\]
Let $\delta\to 0+$, we get
\[
\varphi^{\An}(\ord_E)\geq \lim_{j\to\infty}\varphi^{j,\An}(\ord_E)\,.
\]
The converse is trivial.
\end{proof}
\begin{remark}
For readers familiar with the non-Archimedean language of \cite{BJ18b}, our proof in fact implies the following stronger result: $\varphi^{\An}$ extends uniquely to a function in $\PSH^{\NA}(L)$ and $\varphi^{j,\An}\to \varphi^{\An}$ in $\PSH^{\NA}(L)$. See \cite[Theorem~4.28, Corollary~4.58]{BJ18b}.
\end{remark}

\subsection{Singularity divisors}
Let $X$ be a projective manifold of dimension $n$. Let $L$ be a semi-ample line bundle with a smooth non-negatively curved Hermitian metric $h$. Let $\omega=c_1(L,h)$. 

\begin{definition}\label{def:sdd}
Let $\psi\in \PSH(X,\omega)$.
 Let $\pi:Y\rightarrow X$ be a birational morphism from a normal $\mathbb{Q}$-factorial projective variety. Define the \emph{singularity divisor} of $\psi$ on $Y$ as
 \[
 \Div_Y \psi:=\sum_{E} \nu_E(\psi)\,E\,,
 \]
 where $E$ runs over the set of prime divisors on $Y$, $\nu_E(\psi)$ is the generic Lelong number of $\pi^*\psi$ along $E$. Note that this is a countable sum by Siu's semi-continuity theorem.
 
 Let $D$ be an effective $\mathbb{R}$-divisor on $Y$. We say that the singularities of $\psi$ are \emph{determined} on $Y$ by $D$ if for any birational model $\Pi:Z\rightarrow Y$, $\Div_Z \psi=\Pi^*D$.
\end{definition}
We can regard $\Div_Y\psi$ as the divisorial part of Siu's decomposition of $\ddc\pi^*\psi$. 

\begin{remark}\label{rmk:divNS}
In general, a divisor with countably many components does not define a class in the N\'eron--Severi group, but in the case of $\Div_Y \psi$, this can be easily defined. In fact, write $\Div_Y \psi=\sum_{i=1}^{\infty} a_iD_i$. Here we allow $a_i$ to be $0$. Clearly, $\pi^*L-\sum_{i=1}^r a_i D_i$ is pseudo-effective. It follows that $\sum_{i=1}^{\infty} a_i D_i$ converges as a sum in the N\'eron--Severi group $\mathrm{NS}^1(Y)\otimes \mathbb{R}$, see \cite[Proposition~1.3]{BFJ09}. 
In particular, we can talk about the intersection between $\Div_Y \psi$ and divisors.
\end{remark}

As a consequence of resolution of singularities, any potential with analytic singularity admits a model where its singularities are determined (\cite[Lemma~2.3.19]{MM07}).

\begin{definition}\label{def:ext}
Let $E$ be a prime divisor over $X$. An \emph{extraction} of $E$ is a proper birational morphism $\pi:Y\rightarrow X$ from a \emph{normal $\mathbb{Q}$-factorial variety} $Y$, such that $E$ is a prime divisor on $Y$ and that $-E$ is $\pi$-ample.

If there is an extraction of $E$, we call $E$ an \emph{extractable} divisor.
\end{definition}
Observe that when $X$ is Fano, an extractable divisor $E$ is dreamy in the sense that the doubly graded algebra
\begin{equation}\label{eq:dreamydef}
\bigoplus_{m\in \mathbb{Z}_{\geq 0}}\bigoplus_{p\in \mathbb{Z}}H^0(Y,-m\pi^*K_X-pE)
\end{equation}
is finitely generated.

In general, when the log discrepancy of $E$ is well-behaved, one can run a suitable MMP to extract $E$. See \cite[Corollary~1.4.3]{BCHM10}, \cite[Section~1.4]{Kol13} for details.

Assume that $L$ is ample.
Let $F$ be an extractable divisor. Let $\pi:Y\rightarrow X$ be an extraction of $F$. We can take $A\in \mathbb{Q}_{>0}$ large enough, so that $A\pi^*L-F$ is semi-ample. In particular, take $B$ large enough, so that $B(A\pi^*L-F)$ is base-point free. Take a basis $s_1,\ldots,s_N$ of $H^0(X,B(A\pi^*L-F))$. Let
\[
\psi=\frac{1}{AB}\log \max_{i=1,\ldots,N} |s_j|_{h^{AB}}^2\,.
\]
Then the singularities of $\psi$ are determined on $Y$ by $A^{-1}F$ (see \cref{def:sdd}).

\subsection{Quasi-analytic singularities}
Let $X$ be a projective manifold of dimension $n$. Let $L$ be a big and semi-ample line bundle on $X$. Let $h$ be a smooth non-negatively curved Hermitian metric on $L$. Let $\omega=c_1(L,h)$.

\begin{definition}\label{def:qasing}
 We say a potential $\varphi\in \PSH(X,\omega)$ has \emph{quasi-analytic singularities} if there is a birational model $\pi:Y\rightarrow X$, a snc $\mathbb{R}$-divisor $D$ on $Y$, such that the singularities of $\psi$ are determined on $Y$ by $D$ (see \cref{def:sdd}). In this case, we also say that $\varphi$ has quasi-analytic singularities along $D$.
\end{definition}

\begin{lemma}\label{lma:qamis}
Let $\varphi\in \PSH(X,\omega)$ be a potential with quasi-analytic singularities along a snc $\mathbb{Q}$-divisor $D$ on a birational model $\pi:Y\rightarrow X$, then
\[
\mathscr{I}(k\pi^*\varphi)=\mathcal{O}_Y(-\floor{kD})
\]
for any $k\in \mathbb{Q}_{>0}$.
\end{lemma}
\begin{proof}
Without loss of generality, we take $k=1$.
Recall that we have assumed that the model is projective. Take a sufficiently ample line bundle $H$ on $Y$, so that $H-D$ is semi-ample and $H-\pi^*L$ is ample. Take a $m\in \mathbb{Z}_{>0}$ so that $m(H-D)$ is globally generated. Fix a smooth positively curved metric $h$ on $H$. Let $\omega':=c_1(H,h)$, we may assume that $\omega'>\pi^*\omega$.
Take a basis $s_1,\ldots,s_N$ of $H^0(Y,m(H-D))$. Let
\[
\psi=\frac{1}{m}\log \max_{i=1,\ldots,N} |s_i|^2_{h^m}\,.
\]
Then we know that
\[
\mathscr{I}(\psi)=\mathcal{O}_Y\left(-\floor{D}\right)\,.
\]
But we know that $\pi^*\varphi\sim_{\mathscr{I}}\psi$ as $\omega'$-psh functions, so we conclude. 
\end{proof}
\begin{remark}
    We rephrase the proof of \cref{lma:qamis} in fancier terms: Let
    \[
    \PSH^{\Mdl}(X):=\varinjlim_{\omega}\PSH^{\Mdl}(X,\omega)\,,
    \]
    where $\omega$ runs over all K\"ahler forms on $X$, when $\omega\leq \omega'$, the map $\PSH^{\Mdl}(X,\omega)\rightarrow \PSH^{\Mdl}(X,\omega')$ is given by the $P_{\omega'}[\bullet]$. We take the filtered colimit in the category of sets. We define a class $[\varphi]\in \PSH^{\Mdl}(X)$ to be analytic if some representative is analytic. Now the proof of \cref{lma:qamis} says that when $\varphi\in\PSH^{\Mdl}(X,\omega)$ is quasi-analytic, the class $[\varphi]\in \PSH^{\Mdl}(X)$ is analytic.
\end{remark}

\begin{lemma}\label{lma:diffnab}
Let $\varphi\in \PSH(X,\omega)$ be a potential with quasi-analytic singularities along a snc $\mathbb{R}$-divisor $D$ on $X$, then $L-D$ is nef. If moreover $\int_X \omega_{\varphi}^n>0$, then $L-D$ is big and nef.
\end{lemma}
\begin{proof}
Consider the positive current $\omega_{\varphi}-[D]$ in $c_1(L-D)$. Take a quasi-equisingular approximation $h_j$ of $\omega_{\varphi}-[D]$ (\cite{Cao14}). The Lelong number condition and the fact that $h_j$ has analytic singularities show that its local potential is in fact bounded. Hence $L-D$ is nef. Now the assumption $\int_X \omega_{\varphi}^n>0$ implies that $(L-D)^n>0$, hence $L-D$ is big (\cite{DP04}). 
\end{proof}

\subsection{Non-archimedean envelopes}
Let $X$ be a compact Kähler manifold of dimension $n$. Let $L$ be an ample line bundle with a smooth strictly positively curved metric $h$. Let $\omega=c_1(L,h)$.

Let $\mathbf{v}=(v_1,\ldots,v_m)$ be a valuation of $\mathbb{C}(X)$ with value in $\mathbb{R}^m$. We assume for simplicity that each $v_i$ is divisorial.

\begin{definition}
 Let $\mathbf{a}=(a_1,\ldots,a_m)\in \mathbb{R}_{\geq 0}^m$. Define a potential in $\PSH(X,\omega)\cup\{-\infty\}$:
 \begin{equation}\label{eq:psigeqa1}
     \psi_{\mathbf{v}\geq \mathbf{a}}:=\sups\left\{\,\psi\in \PSH(X,\omega):\psi\leq 0, v_i(\psi)\geq a_i  \text{ for }i=1,\ldots,m \, \right\}\,.
 \end{equation}
 We also define
 \begin{equation}\label{eq:psigeqa}
     \psi'_{\mathbf{v}\geq \mathbf{a}}:=\sups_{k\in \mathbb{Z}_{>0}} \frac{1}{k}\sups\left\{\,\log|s|_{h^k}^2:s\in H^0(X,L^k),\sup_X |s|_{h^k}\leq 1, v_i(s)\geq ka_i \text{ for }i=1,\ldots,m \,\right\}\,.
 \end{equation}
\end{definition}
Observe that $\psi_{\mathbf{v}\geq \mathbf{a}}$ itself is a candidate in the sup in \eqref{eq:psigeqa1}, provided that $\psi_{\mathbf{v}\geq \mathbf{a}}\neq -\infty$. Obviously, $\psi_{\mathbf{v}\geq \mathbf{a}}$ is either $\mathscr{I}$-model or $-\infty$. 

\begin{lemma}\label{lma:IMDL}Assume that $\psi_{\mathbf{v}\geq \mathbf{a}}$ has positive mass, then $P[\psi'_{\mathbf{v}\geq \mathbf{a}}]_{\mathscr{I}}= \psi_{\mathbf{v}\geq \mathbf{a}}$.
\end{lemma}
\begin{proof}
That $P[\psi'_{\mathbf{v}\geq \mathbf{a}}]_{\mathscr{I}}\leq \psi_{\mathbf{v}\geq \mathbf{a}}$ is trivial, we prove the converse. Write $\psi=\psi_{\mathbf{v}\geq \mathbf{a}}$.
It suffices to show that for any small enough $\epsilon>0$, any fixed divisorial valuation $v$ of $\mathbb{C}(X)$, we can construct a section $s\in H^0(X,L^k)$ for some large $k$, so that $k^{-1}v_i(s)\geq  v_i(\psi)$ for all $i$ and $k^{-1}v(s)\leq v(\psi)+\epsilon$. We may assume that $a_i>0$ for all $i$. Let $b_i=v_i(\psi)$. Then $b_i\geq a_i$. Let $b=v(\psi)$. 

By \cite[Lemma~4.4]{DDNLmetric}, we can construct a potential $\psi'$, more singular than $\psi$, such that
\[
v_i(\psi')>b_i\,,\quad v(\psi')\leq v(\psi)+\epsilon\,.
\]

Take a small enough $\delta\in \mathbb{Q}_{>0}$, such that \[
(1+\delta)^{-1}v_i(\psi')> b_i
\]
for all $i$.
Regarding $\psi'$ as a metric on $(1+\delta)L$ and applying \cite[Corollary~13.23]{Dem12} and its proof, we find a sequence of sections $s_k\in H^0(X,L^k)$ for some sequence $k$ increasing to $\infty$, such that
\[
\frac{1+\delta}{k}[\Div s_k]\to \delta\omega+\omega_{\psi'}\,,\quad \frac{1+\delta}{k}v_i(s_k)\to v_i(\psi')\,.
\]
Thus for $k$ large enough, 
\[
\frac{1}{k}v(s_k)\leq b+\epsilon\,,\quad  \frac{1}{k}v_i(s_k)>b_i\,.
\]
\end{proof}

As a particular case, let $v_i=c_i\ord_{F_i}$ be a $\mathbb{Q}$-divisorial valuation, $a_i\in \mathbb{R}$ ($i=1,\ldots,m$). We have
\[
\psi'_{\mathbf{v}\geq \mathbf{a}}=\sups_{k\in \mathbb{Z}_{>0}\text{ sufficiently divisible}}\frac{1}{k}\sups\Bigg\{\,\log|s|_{h^k}^2: s\in H^0(X,kL-\sum_{i=1}^mk a_i c_i^{-1}F_i),
\sup_X |s|_{h^k}\leq 1\,\Bigg\}\,.
\]

Let $D$ be an effective $\mathbb{Q}$-divisor on $X$. In this paper, $\mathbb{Q}$-divisor are allowed to have countably many components.
If $D$ has finitely many irreducible components, say $D=\sum_{i=1}^r a_i D_i$, we define
\begin{equation}
\psi_{\geq D}:=\psi_{(\ord_{D_i})\geq (a_i)}\,.
\end{equation}
In general, if $D$ has countably many components, say $D=\sum_{i=1}^{\infty} a_i D_i$, we just let
\[
\psi_{\geq D}:=\inf_{j=1,\ldots,\infty} \psi_{\geq \sum_{i=1}^{j} a_i D_i}\,.
\]

\subsection{Extended deformation to the normal cone}\label{subsec:extdef}
Let $\psi\in \PSH(X,\omega)$ be a potential with analytic singularities. Let $\pi:Y\rightarrow X$ be a log resolution of $\psi$. Let $\Psef(\psi)$ be the pseudo-effective threshold of $\psi$. Namely,
\[
\Psef(\psi)=\sup\left\{\,t\geq 0: \pi^*L-t\Div_Y\psi \text{ is pseudo-effective}\,\right\}\,.
\]
We define a test curve $\psi^+_{\bullet}$ as follows (see \cref{sec:rwn} for the general theory of test curves):
\[
\psi^+_{\tau}:=
\left\{
\begin{aligned}
0\,,&\quad \tau\leq 0\,,\\
\psi_{\geq \tau\Div_Y\psi}\,,&\quad \tau\in 0<\tau\leq \Psef(\psi)\,,\\
-\infty\,,&\quad \tau>\Psef(\psi)\,.
\end{aligned}
\right.
\]
We call this construction the \emph{extended deformation to the normal cone} with respect to $\psi$. See \cref{lma:genext} for an explanation of this terminology.

This construction can be realized geometrically. 
\begin{definition}\label{def:dre}
 Let $\psi\in \PSH(X,C\omega)$ ($C\in \mathbb{Z}_{>0}$) be a potential with analytic singularities.  Let $\pi:Y\rightarrow X$ be a log resolution. 
Let $A>0$ be an integer so that $A\Div_Y\psi$ is integral. We say $\psi$ is \emph{dreamy} if the double-graded ring
 \[
 R(X,L,\psi):=\bigoplus_{k\in \mathbb{Z}_{\geq 0}} \bigoplus_{s\in \mathbb{Z}_{\geq 0}} H^0(Y,kACL-sA\Div_Y\psi)
 \]
 is finitely generated.
\end{definition}
Note that whether or not $\psi$ is dreamy  does not depend on the choice of $A$. 

Assume that $\psi\in \PSH(X,\omega)$ is dreamy and $L$ is ample.
Let $(\mathcal{X},\mathcal{L})$ be the relative proj of 
\[
\bigoplus_{k\in \mathbb{Z}_{\geq 0}} \bigoplus_{s\in \mathbb{Z}_{\geq 0}}t^{-s} H^0(Y,kAL-sA\Div_Y\psi)
\]
over $\mathbb{C}$. Then $(\mathcal{X},\mathcal{L})$ is a test configuration of $(X,L^A)$. It follows from \cref{lma:IMDL} that the corresponding test curve is just $\psi^+$.

More generally, let $\psi\in \PSH(X,\omega)$. We define
$\Psef(\psi)$ as the sup of $t\geq 0$, such that on each birational model $\pi:Y\rightarrow X$,
$\pi^*L-t\Div_Y\psi$ is pseudo-effective. 
This definition coincides with the previous one when $\psi$ has analytic singularities.

We define the corresponding test curve $\psi^+_{\bullet}$ as follows:
When $\tau\leq 0$, set $\psi^+_{\tau}=0$.
When $0<\tau<\Psef(\psi)$, we define
\[
\psi^+_{\tau}=\lim_{Y}\psi_{\geq \tau \Div_Y\psi}\,,
\]
where the limit is a limit of decreasing net taken over all birational models $\pi:Y\rightarrow X$. Note that the limit is $\mathscr{I}$-model by \cite[Lemma~2.20]{DX22} (Strictly speaking, \cite[Lemma~2.20]{DX22} only deals with decreasing sequences, but the proof works for decreasing nets as well). Define
\[
\psi^+_{\Psef(\psi)}=\lim_{\tau\to \Psef(\psi)-}\psi^+_{\tau}
\]
and
\[
\psi^+_{\tau}=-\infty
\]
if $\tau>\Psef(\psi)$.

\subsection{Generalized deformation to the normal cone}\label{subsec:defn}
Let $\psi\in \PSH(X,\omega)$ be a model potential with analytic singularities. 
We define a test curve (see \cref{sec:rwn} for the precise definition) $\psi_{\bullet}$ by 
\[
\psi_{\tau}:=
\left\{
\begin{aligned}
0\,,&\quad \tau\leq -1\,,\\
P[(1+\tau)\psi]\,,&\quad \tau\in (-1,0]\,,\\
-\infty\,,&\quad \tau>0\,.
\end{aligned}
\right.
\]
The test curve $\psi_{\bullet}$ is a truncated version of $\psi^+_{\bullet}$:
\begin{lemma}\label{lma:genext}
When $\tau\in [0,1]$, $\psi^+_{\tau}=\psi_{\tau-1}$.
\end{lemma}

The test curve $\psi_{\bullet}$ and its associated geodesic ray were studied in \cite{Dar17} and \cite{DDNLmetric}.

The following result due to Darvas (\cite{Da17}) characterizes the geodesic ray induced by $\psi_{\bullet}$. 
\begin{proposition}
	Let $\psi\in \PSH^{\Mdl}(X,\omega)$, then $\check{\psi}_{t}$ ($t\geq 0$) is the increasing limit of $\ell^k_t$, where $(\ell^k_t)_{t\in [0,-E(\max\{-k,\psi\})]}$ is the geodesic from $0$ to $\max\{-k,\psi\}$.  
\end{proposition}

Assume that $\psi$ has analytic singularities along a $\mathbb{Z}$-divisor $\Div_X \psi$ on $X$ and that $L-\Div_X\psi$ is semi-ample.
In this case, let $\mathcal{X}=\Bl_{\Div_X \psi\times\{0\}} X\times \mathbb{C}$ be the deformation to the normal cone. Let $\mathcal{E}$ be the exceptional divisor.
Let $\Pi:\mathcal{X}\rightarrow X\times \mathbb{C}$ be the natural map and let $p_1:X\times \mathbb{C} \rightarrow X$ be the natural projection. Let $\mathcal{L}=\Pi^*p_1^*L\otimes \mathcal{O}_{\mathcal{X}}(-\mathcal{E})$.
Then we have a test configuration $(\mathcal{X},\mathcal{L})$ of $(X,L)$. By \cref{ex:blflag}, the test curve induced by the filtration of this test configuration is exactly $\psi_{\bullet}$. 
Note that $\psi_{\bullet}$ is induced by the filtration in \cref{ex:def}.

\subsection{Entropy and delta invariant}
Let $X$ be a compact Kähler manifold of dimension $n$. Let $L$ be an ample line bundle. Let $\omega\in c_1(L)$ be a K\"ahler form.

We recall that for an $\mathbb{R}$-Weil divisor $D=\sum_{i} a_i D_i$ with $a_i\neq 0$, $D_i$ prime and pairwise distinct, $\Redu D:=\sum_i D_i$.

\begin{definition}\label{def:naentgeneral}
 Let $\psi\in \PSH(X,\omega)$. We define the \emph{entropy} of $[\psi]$ as
 \[ 
    \Ent([\psi]):=\frac{n}{V}\varliminf_{Y} \left(\left\langle\pi^*L-\Div_Y\psi\right\rangle^{n-1} \cdot (K_{Y/X}+\Redu \Div_Y \psi) \right)\in [0,\infty]\,,
 \]
 where $\pi:Y\rightarrow X$ runs over all birational models on $X$.  Here the product $\langle \bullet \rangle$ is the movable intersection in the sense of \cite{BFJ09}, \cite{Bou02}.
 We formally set $\Ent([-\infty])=0$. 
\end{definition}
We observe that $\Ent([\psi])$ depends only on the $\mathscr{I}$-singularity type of $\psi$. To the best of the author's knowledge, this invariant has never been defined in the literature.
\begin{remark}
    The condition $\psi\in \PSH(X,\omega)$ is not essential. We can define the same quantity for any quasi-psh function.
\end{remark}

We can now define our new delta invariant: 
\begin{definition}\label{def:delta2}
 We define the \emph{pluripotential-theoretic $\delta$-invariant} as
 \[
 \delta_{\mathrm{pp}}=\inf_{\psi}\frac{\int_{-\infty}^{\infty}\Ent([\psi^+_{\tau}])\,\mathrm{d}\tau}{nV^{-1}\int_{-\infty}^{\infty} \left(\int_X \omega\wedge \omega_{\psi^+_{\tau}}^{n-1}-\int_X\omega_{\psi^+_{\tau}}^n\right)\,\mathrm{d}\tau}\,,
 \]
 where $V=(L^n)$, $\psi$ runs over the set of $\omega$-psh functions with some non-zero Lelong number on $X$. The quotient depends only on the $\mathscr{I}$-singularity type of $\psi$.
\end{definition}
\begin{remark}\label{rmk:deltappImdoeldep}
We remark that $\psi_{\bullet}^+$ depends only on the $\mathscr{I}$-singularity type of $\psi$, hence in the definition above, it suffices to take $\mathscr{I}$-model potentials $\psi$. 
\end{remark}

For the next definition, we need to introduce some notations.
We define a polynomial
\begin{equation}\label{eq:Gn}
G_{n-1}(A,B)=\sum_{j=0}^{n-1}\frac{1}{j+1}\binom{n-1}{j}(-1)^j \left(A^{n-1-j}\cdot B^j \right)=\frac{1}{nB}\left(A^n-(A-B)^n \right)\,.
\end{equation}
When $A$, $B$ are divisors on $X$, $G_{n-1}(A,B)$ is considered as an element in the Chow ring of $X$.
We observe that when $A,B_1,B_0\in \mathbb{R}$ and if we set $B_t=tB_1+(1-t)B_0$ ($t\in [0,1]$), then
\begin{equation}\label{eq:trivialint}
\int_0^1 (A-B_t)^{n-1}\,\mathrm{d}t=G_{n-1}(A-B_0,B_1-B_0)\,.
\end{equation}

\begin{definition}\label{def:delta}
We define the \emph{$\delta'$-invariant} of $(X,L)$ as
 \[
 \delta':=\inf_{\psi}\frac{(K_{Y/X}\cdot (-\Div_Y\psi)^{n-1})+n\left(G_{n-1}(L,\Div_Y\psi)\cdot \Redu \Div_Y\psi\right)}{n\int_0^1 \left(\int_X \omega\wedge \omega_{\tau\psi}^{n-1}-\int_X  \omega_{\tau\psi}^{n}\right)\,\mathrm{d}\tau}\,,
 \]
 where $\pi:Y\rightarrow X$ is a log resolution of $\psi$.
 Here $\psi$ runs over the set of unbounded $\omega$-psh functions with analytic singularities. The quotient depends only on the singularity type of $\psi$.
\end{definition}

\section{Preliminaries}\label{sec:pre}
Let $X$ be a compact Kähler manifold of dimension $n$. Let $\omega$ be a Kähler form on $X$. We introduce a number of functionals on the space of K\"ahler potentials and on the space of geodesic rays.
\subsection{Archimedean functionals}
In this section, we recall the definitions of several functionals in Kähler geometry. For the definition of $\mathcal{E}^1=\mathcal{E}^1(X,\omega)$, we refer to \cite{Dar19} and references there in. We write $\mathcal{E}^{\infty}(X,\omega)$ for the set of bounded potentials in $\PSH(X,\omega)$.

Define $V=V_{\omega}:=\int_X \omega^n$.
Let $E:\mathcal{E}^1\rightarrow \mathbb{R}$ denote the Monge--Ampère energy functional:
\[
E(\varphi)=\frac{1}{V}\sum_{j=0}^n\int_X \varphi\, \omega_{\varphi}^j \wedge \omega^{n-j}\,.
\]
For $\varphi\in \mathcal{E}^1(X,\omega)$, define
\[
\Ent(\varphi):=
\left\{
\begin{aligned}
\frac{1}{V}\int_X \log\left(\frac{\omega_{\varphi}^n}{\omega^n} \right)\omega_{\varphi}^n &\,,\quad \text{if }\omega_{\varphi^n}\text{ is absolutely continuous with respect to } \omega^n\,,\\
\infty&\,,\quad \text{otherwise}\,.
\end{aligned}
\right.
\]
Let $\alpha$ be a smooth real $(1,1)$-form on $X$. We define the functional $E^{\alpha}:\mathcal{E}^1\rightarrow \mathbb{R}$ by
\[
E^{\alpha}(\varphi):=\frac{1}{nV}\sum_{j=0}^{n-1}\int_X \varphi\,\alpha\wedge \omega_{\varphi}^j\wedge \omega^{n-1-j}\,.
\]
In particular, the Ricci energy is defined as
\[
E_R:=E^{-n\Ric \omega}=-\frac{1}{V}\sum_{j=0}^{n-1}\int_X \varphi\,\Ric \omega\wedge \omega_{\varphi}^j\wedge \omega^{n-1-j}\,.
\]
Define the $\tilde{J}$ functional as  $I-J$, namely
\begin{equation}\label{eq:Jtil}
\tilde{J}(\varphi)=E(\varphi)-\frac{1}{V}\int_X \varphi\,\omega_{\varphi}^n\,.
\end{equation}
Note that
\begin{equation}\label{eq:EomegaandE}
E^{\omega}=E+\frac{1}{n}\tilde{J}\,.
\end{equation}

Let $M:\mathcal{E}^1\rightarrow (-\infty,\infty]$ denote the Mabuchi functional:
\[
M(\varphi)=\bar{S}E(\varphi)+\Ent(\varphi)+E_R(\varphi)\,,
\]
where $\bar{S}$ is the average scalar curvature.
In this paper, it is convenient to use a different normalization of the Mabuchi functional, so we define the \emph{twisted Mabuchi functional} $\tilde{M}:\mathcal{E}^1(X,\omega)\rightarrow (-\infty,\infty]$ as
\[
\tilde{M}:=M-\bar{S}E=\Ent+E_R\,.
\]

Now assume that $[\omega]=c_1(L)$ for some ample line bundle $L$ on $X$. Fix a smooth Hermitian metric $h$ on $L$ with $c_1(L,h)=\omega$.

For any $k\in \mathbb{Z}_{>0}$, the Donaldson's $\mathscr{L}_k$-functional (\cite{Don05}) is defined as
\[
\mathscr{L}_k(\varphi):=-\frac{2}{kV} \log \frac{\det\|\cdot\|_{\Hilb_k(\varphi)}}{\det\|\cdot\|_{\Hilb_k(0)}}\,.
\]
Here $\Hilb_k(\varphi)$ is the norm on $H^0(X,K_X\otimes L^k)$ defined by
\[
\|s\|^2_{\Hilb_k(\varphi)}=\int_X \left(s,\bar{s}\right)_{h^k}e^{-k\varphi}\,.
\]
\begin{definition}
	Here our convention of the determinant follows that in \cite{BE21}, which differs from the convention of \cite{DX22} by a factor of $2$.
\end{definition}

\begin{theorem}\label{thm:Bernconv}
	For each $k\geq 1$, the functional $\mathscr{L}_k$ is convex along finite energy geodesics in $\mathcal{E}^1$.
\end{theorem}
This result is essentially Berndtsson's convexity theorem (\cite{Ber09b}, \cite{Ber09})). See \cite[Proposition~2.12]{DLR20} for details.

\subsection{Radial functionals}
In this section, we assume that the Kähler class $[\omega]$ is in the integral Néron--Severi group. Take an ample line bundle $L$ on $X$ so that $[\omega]=c_1(L)$. Fix a smooth positive metric $h$ on $L$ with $c_1(L,h)=\omega$.

Let $\mathcal{R}^1(X,\omega)$ be the space of $\mathcal{E}^1(X,\omega)$ geodesic rays emanating from $0$. That is, a general element $\ell\in \mathcal{R}^1$ is a map $[0,\infty)\rightarrow \mathcal{E}^1$, such that $\ell_0=0$ and such that $\ell|_{[0,A]}$ is a (finite energy) geodesic in $\mathcal{E}^1$ for any $A>0$. See \cite{DL20} for details. 

We also write $\mathcal{R}^{\infty}(X,\omega)$ for the set of locally bounded geodesic rays emanating from $0$.

For $F=\Ent,E^{\alpha},M,\tilde{M},E,\tilde{J}$, we define a corresponding radial functional $\mathbf{F}$ on $\mathcal{R}^1$ by 
\[
\mathbf{F}(\ell):=\lim_{t\to\infty}\frac{1}{t}F(\ell_t)\,.
\]
For each of them, the limit is well-defined by \cite[Proposition~4.5, Theorem~4.7]{BDL17}.

\subsection{Non-Archimedean functionals}\label{subsec:Xan}
We write $X^{\An}$ for the Berkovich analytification of $X$ with respect to the trivial valuation on $\mathbb{C}$. As a set, $X^{\An}$ consists of all real semi-valuations (up to equivalence) extending the trivial valuation on $\mathbb{C}$. There is a natural topology known as the Berkovich topology on $X^{\An}$. We always endow $X^{\An}$ with this topology. 
There is a continuous morphism of locally ringed spaces from $X^{\An}$ to $X$ with the Zariski topology.
Let $L^{\An}$ be the pull-back of $L$ to $X^{\An}$. See \cite[Section~3.5]{Berk12}. We refer to \cite{BJ18b} for the definition of $\mathcal{E}^{1,\NA}=\mathcal{E}^1(X^{\An},L^{\An})$. For $\ell\in \mathcal{R}^1$, we write $\ell^{\NA}$ for the corresponding potential in $\mathcal{E}^{1,\NA}$ in the sense of \cite{BBJ15}, namely
\[
\ell^{\NA}(v):=-G(v)(\Phi)\,,
\]
where $G(v)$ is the Gauss extension of $v$ and $\Phi$ is the potential on $X\times \Delta$ corresponding to $\ell$. Recall that there is a natural embedding $\mathcal{E}^{1,\NA}\hookrightarrow \mathcal{R}^1$. We will often use this embedding implicitly. Geodesic rays in the image of this embedding are known as \emph{maximal} geodesic rays.

For $F=E,E^{\alpha},\tilde{J}$, we write
\[
F^{\NA}(\ell^{\NA}):=\mathbf{F}(\ell)\,,
\]
when $\ell$ is a maximal geodesic ray.
For explanation of this terminology, see \cite{BHJ16} and \cite[Proposition~2.38]{Li20}. We also remark that $\tilde{J}^{\NA}$-functional appeared already in \cite{Der16} under the name of the minimum norm.

Let $\psi\in \mathcal{E}^{1,\NA}$, we write
\[
\Ent^{\NA}(\psi)=\Ent^{\NA}\left(\MA(\psi)\right):=\frac{1}{V}\int_{X^{\An}} A_X\,\MA(\psi)\,,
\]
where $A_X:X^{\An}\rightarrow [0,\infty]$ denotes the log discrepancy functional (see \cite{JM12}) and $\MA(\psi)$ denotes the Chambert-Loir measure (see \cite{CLD12}, \cite{CL06}, \cite{BJ18b}). 
We also write
\[
\tilde{M}^{\NA}=E_R^{\NA}+\Ent^{\NA}\,.
\]
In the case of $\mathscr{L}_k$ and $\ell$ is maximal, we write 
\begin{equation}\label{eq:defLkr}
\mathscr{L}_k^{\NA}(\ell^{\NA})=\lim_{t\to\infty}\frac{1}{t}\mathscr{L}_k(\ell_t)\,.
\end{equation}

Recall the definition of $\delta$-invariant:
\begin{equation}\label{eq:deltadef}
\delta=\delta([\omega]):=\inf_{v\in \Val_X^*}\frac{A_{X}(v)}{S_L(v)}\,,
\end{equation}
where $\Val_X^*$ denotes the space of non-trivial real valuations of $\mathbb{C}(X)$ (\cite{JM12}) and 
\[
S_L(v):=\int_0^{\infty}\vol(L-tv)\,\mathrm{d}t
\]
and
\[
\vol(L-tv)=\lim_{k\to\infty} \frac{n!}{k^n}h^0(X,L^k\otimes \mathfrak{a}_{tk})
\]
with $\mathfrak{a}_{t}$ being the ideal sheaf defined by the condition that $v\geq t$. Recall that in the Fano setting, there is always a quasi-monomial valuation that achieves the minimum in \eqref{eq:deltadef} (see \cite[Theorem~4.20]{Xu21}, \cite{BLZ19})).
Recall that (\cite[Section~2.9, Theorem~5.16]{BJ18})
\begin{equation}\label{eq:de1}
\delta([\omega])=\inf_{\mu\in \mathcal{M}(X^{\An})}\frac{\Ent^{\NA}(\mu)}{E^*(\mu)}\,,
\end{equation}
where $\mathcal{M}(X^{\An})$ denotes the set of Radon measures on $X^{\An}$ with total mass $V$,
\[
E^*(\mu):=\sup_{\psi\in \mathcal{E}^1(X^{\An},L^{\An})} \left(E(\psi)-\int_{X^{\An}}\psi\,\mathrm{d}\mu\right)\,.
\]
It is easy to see that for $\varphi\in \mathcal{E}^{1,\NA}$,
\begin{equation}\label{eq:EstJt}
E^*(\MA(\varphi))=\tilde{J}^{\NA}(\varphi)\,.
\end{equation}

\subsection{Flag ideals and test configurations}
Let $X$ be a compact Kähler manifold of dimension $n$. Let $L$ be a  big and semi-ample line bundle on $X$. Let $h$ be a smooth, non-negatively curved metric on $L$. Let $\omega=c_1(L,h)$.
\begin{definition}
	A \emph{flag ideal} on $X\times \mathbb{C}$ is a $\mathbb{C}^*$-invariant coherent ideal sheaf of $X\times \mathbb{C}$ that is cosupported on the central fibre. Equivalently, a flag ideal is an ideal of the form
	\begin{equation}
	\mathscr{I}=I_0+I_1t+\cdots+I_{N-1}t^{N-1}+(t^N)\,,
	\end{equation}
	where $I_0\subseteq I_1\subseteq \cdots \subseteq I_{N-1}\subseteq I_N=\mathcal{O}_X$ are coherent ideal sheaves on $X$, $t$ is the variable on $\mathbb{C}$.
\end{definition}

\begin{definition}\label{def:tc}
	A \emph{test configuration} of $(X,L)$ consists of a pair $(\mathcal{X},\mathcal{L})$ consisting of a variety $\mathcal{X}$ and a semi-ample $\mathbb{Q}$-line bundle $\mathcal{L}$ on $\mathcal{X}$,  a morphism $\Pi:\mathcal{X}\rightarrow \mathbb{C}$, a $\mathbb{C}^*$-action on $\mathcal{X},\mathcal{L}$ and an isomorphism $(\mathcal{X}_1,\mathcal{L}|_{\mathcal{X}_1})\cong (X,L)$, so that
	\begin{enumerate}
		\item $\pi$ is $\mathbb{C}^*$-equivariant.
		\item The fibration $\pi$ is equivariantly isomorphic to the trivial fibration $(X\times \mathbb{C}^*,p_1^*L)$ through an isomorphism that extends the given one over $1$. Here $p_1$ denotes the projection to the first factor.
	\end{enumerate}
	A test configuration $(\mathcal{X},\mathcal{L})$ can be compactified by gluing the trivial fibration over $\mathbb{P}^1\setminus\{0\}$. We write $(\bar{\mathcal{X}},\bar{\mathcal{L}})$ for the compactified test configuration. We will frequently omit the bars when we talk about compactified test configurations.
	
\end{definition}

\begin{definition}[Donaldson--Futaki invariant]
Let $(\mathcal{X},\mathcal{L})$ be a test configuration of $(X,L)$. 
Take $r\in \mathbb{Z}_{>0}$ so that $\mathcal{L}^r$ is integral.
For $k\in \mathbb{Z}_{>0}$, define $w(rk)$ as the weight of the $\mathbb{C}^*$-action on $H^0(\mathcal{X}_0,\mathcal{L}^{rk}|_{\mathcal{X}_0})$. By equivariant Riemann--Roch theorem, we can write
\[
w(rk)=a(rk)^{n+1}+b(rk)^n+\mathcal{O}(k^{n-1})\,.
\]
Define the \emph{twisted Donaldson--Futaki invariant} of $(\mathcal{X},\mathcal{L})$ as
\[
\widetilde{\DF}(\mathcal{X},\mathcal{L})=-2b \frac{n!}{V}\,.
\]
\end{definition}
Let $\ell$ be the Phong--Sturm geodesic ray associated to $(\mathcal{X},\mathcal{L})$ and let $\phi=\ell^{\NA}\in \mathcal{H}^{\NA}$ be the non-Archimedean potential defined by $(\mathcal{X},\mathcal{L})$.
\begin{proposition}[{\cite[Proposition~2.8]{BHJ16}},{\cite[Theorem~5.3]{Li20}}]\label{prop:BHJL} Assume that $L$ is ample.
\begin{enumerate}
    \item Let $\ell\in \mathcal{E}^{1,\NA}$, then 
    \[
    \tilde{M}^{\NA}(\ell^{\NA})\leq \tilde{\mathbf{M}}(\ell)\,,\quad \Ent^{\NA}(\ell^{\NA})\leq \mathbf{Ent}(\ell)\,.
    \]
    Equality holds if $\ell$ is the Phong--Sturm geodesic ray of some test configuration.
    \item 
    Let $(\mathcal{X},\mathcal{L})$ be a (not necessarily normal) test configuration of $(X,L)$. Let $p:\tilde{\mathcal{X}}\rightarrow \mathcal{X}$ be the normalization. Let $\tilde{L}=p^*\mathcal{L}$.
    Then
     \begin{equation}\label{eq:tildDF}
        \tilde{\mathbf{M}}(\ell)=\tilde{M}^{\NA}(\phi)=\widetilde{\DF}(\mathcal{X},\mathcal{L})-\frac{1}{V}\left(\left(\tilde{X}_0- \tilde{X}_0^{\mathrm{red}}\right)\cdot \tilde{\mathcal{L}}^n\right)\,.
    \end{equation}
\end{enumerate}
 \end{proposition}
The intersection-theoretic formulae of the Donaldson--Futaki invariant were obtained first in \cite{Oda13} and \cite{Wang12}.

\section{The theory of test curves}\label{sec:rwn}
In this section, we review and extend the theory of test curves. 

\subsection{Ross--Witt Nystr\"om correspondence}
Results in this section are contained in \cite{RWN14}, \cite{DDNL18big} and \cite{DX22}. The references work with ample line bundles and Kähler forms, but the readers can readily check that all arguments work for semi-ample line bundles and real semi-positive forms.

Let $X$ be a compact Kähler manifold of dimension $n$. Let $\omega$ be a real semi-positive form on $X$. Assume that $\int_X \omega^n>0$.
Let $\PSH^{\Mdl}(X,\omega)$ denote the set of model potentials in $\PSH(X,\omega)$.

\begin{definition}\label{def:testcurve}
 A \emph{test curve} is a map $\psi=\psi_{\bullet}:\mathbb{R}\to \PSH^{\Mdl}(X,\omega)\cup\{-\infty\}$, such that
\begin{enumerate}
    \item $\psi_{\bullet}$ is concave in $\bullet$.
    \item $\psi$ is usc as a function $\mathbb{R}\times X\rightarrow [-\infty,\infty)$.
    \item $\lim_{\tau\to-\infty}\psi_{\tau}=0$ in $L^1$.
    \item $\psi_{\tau}=-\infty$ for $\tau$ large enough.
\end{enumerate}
Let $\tau^+:=\inf \{\tau\in \mathbb{R}: \psi_{\tau}=-\infty\}$.
We say $\psi$ is \emph{normalized} if $\tau^+=0$.
The test curve is called \emph{bounded} if $\psi_{\tau}=0$ for $\tau$ small enough. Let $\tau^-:=\sup\{\tau\in \mathbb{R}: \psi_{\tau}=0\}$ in this case.

The set of bounded test curves is denoted by $\TC^{\infty}(X,\omega)$.
\end{definition}
\begin{remark}
We remind the readers that our test curves correspond to \emph{maximal} test curves in the literature. 
\end{remark}
\begin{remark}
    In fact, it is more natural to define a test curve only on the interval $(-\infty,\tau^+)$. But we adopt the traditional definition here to facilitate the comparison with the literature.
\end{remark}

\begin{definition}
The \emph{energy} of a test curve $\psi_{\bullet}$ is defined as
\begin{equation}\label{eq:defE}
\mathbf{E}(\psi_{\bullet}):=\tau^++\frac{1}{V}\int_{-\infty}^{\tau^+} \left(\int_X \omega_{\psi_{\tau}}^n-\int_X\omega^n\right)\,\mathrm{d}\tau\,.
\end{equation} 
A test curve $\psi$ is said to be of \emph{finite energy} if $\mathbf{E}(\psi)>-\infty$. We denote the set of finite energy test curves by $\TC^1(X,\omega)$.
\end{definition}

\begin{proposition}\label{prop:tc}
 Let $\psi_{\bullet}$ be a test curve. Then
 \begin{enumerate}
     \item $\tau\mapsto \int_X \omega_{\psi_{\tau}}^n$ is a continuous function for $\tau\in (-\infty,\tau^+)$.
     \item For any $\tau<\tau^+$, $\int_X \omega_{\psi_{\tau}}^n>0$.
     \item The function $\tau \mapsto \log\int_X \omega_{\psi_{\tau}}^n$ is concave for $\tau\in (-\infty,\tau^+)$. 
 \end{enumerate}
\end{proposition}
\begin{proof}
Part (1) and Part (2) follow from \cite[Lemma~3.9]{DX22}. Part (3) is a consequence of \cite[Theorem~6.1]{DDNL19log} and the monotonicity theorem \cite{WN19}.
\end{proof}

\begin{definition}
Let $\ell\in \mathcal{R}^1(X,\omega)$. The \emph{Legendre transform} of $\ell$ is defined as
\[
\hat{\ell}_{\tau}:=\inf_{t\geq 0} \left(\ell_t-t\tau\right)\,,\quad \tau \in \mathbb{R}\,.
\]
Let $\psi\in \TC^1(X,\omega)$, the \emph{inverse Legendre transform} of $\psi$ is defined as
\[
\check{\psi}_t:=\sup_{\tau\in \mathbb{R}} \left(\psi_{\tau}+t\tau\right)\,,\quad t\geq 0\,.
\]
\end{definition}

\begin{theorem}[{\cite[Theorem~3.7]{DX22}}]
The Legendre transform and inverse Legendre transform establish a bijection from $\mathcal{R}^1(X,\omega)$ to $\TC^1(X,\omega)$. For $\ell\in \mathcal{R}^1(X,\omega)$,
We have $\sup_X \ell_1=\tau^+$ and
$\mathbf{E}(\ell)=\mathbf{E}(\hat{\ell})$.

Moreover, under this correspondence, $\mathcal{R}^{\infty}$ corresponds to the set of bounded test curves. When $\ell\in \mathcal{R}^{\infty}$,$\inf_X \ell_1=\tau^-$.
\end{theorem}
Now assume that $\omega=c_1(L,h)$ for some ample line bundle $L$ and a strictly positively curved smooth Hermitian metric $h$ on $L$. 
\begin{definition}
 An \emph{$\mathscr{I}$-model test curve} is a test curve $\psi_{\bullet}$ such that for every $\tau<\tau^+$, $\psi_{\tau}$ is $\mathscr{I}$-model. The set of $\mathscr{I}$-model test curves of finite energy is denoted by $\TC^1_{\mathscr{I}}(X,\omega)$.
\end{definition}
\begin{theorem}[{\cite[Theorem~3.7]{DX22}}]
The Legendre transform and inverse Legendre transform establish a bijection between $\mathcal{E}^{1,\NA}$ and $\TC^1_{\mathscr{I}}(X,\omega)$.
\end{theorem}

\subsection{Test curves induced by filtrations}
Let $X$ be a compact Kähler manifold of dimension $n$. Let $L$ be a  big and semi-ample line bundle on $X$. Let $h$ be a smooth, non-negatively curved metric on $L$. Let $\omega=c_1(L,h)$.
We use the notation
\[
R(X,L):=\bigoplus_{k\in \mathbb{Z}_{\geq 0}}H^0(X,L^k)\,.
\]
\begin{definition}
A \emph{filtration} on $R(X,L)$ is a decreasing, left continuous, multiplicative $\mathbb{R}$-filtration $\mathscr{F}^{\bullet}$ on the ring $R(X,L)$ which is linearly bounded in the sense that there is $C>0$, so that 
\[
\mathscr{F}^{-k\lambda}H^0(X,L^k)=H^0(X,L^k)\,,\quad \mathscr{F}^{k\lambda}H^0(X,L^k)=0\,,
\]
when $\lambda>C$.

A filtration $\mathscr{F}$ is called a \emph{$\mathbb{Z}$-filtration} if $\mathscr{F}^{\lambda}=\mathscr{F}^{\floor{\lambda}}$ for any $\lambda\in \mathbb{R}$.

A $\mathbb{Z}$-filtration $\mathscr{F}$ is called \emph{finitely generated} if the bigraded algebra
\[
\bigoplus_{\lambda \in \mathbb{Z},k\in \mathbb{Z}_{\geq 0}} \mathscr{F}^{\lambda}H^0(X,L^k)
\]
is finitely generated over $\mathbb{C}$.
\end{definition}

Recall that by \cite{RWN14}, a filtration induces a test curve in the following manner. Let $\mathscr{F}^{\bullet}$ be a filtration. For $\tau\in \mathbb{R}$, define
\begin{equation}\label{eq:testcurvfromfilt}
\psi_{\tau}:=\sups_{k\in \mathbb{Z}_{>0}} k^{-1} \sups\left\{\,\log|s|_{h^k}^2: s\in \mathscr{F}^{k\tau}H^0(X,L^k),\sup_X|s|_{h^k}\leq 1\,\right\}\,.
\end{equation}
By \cite[Theorem~3.11]{DX22}, $\psi_{\tau}$ is $\mathscr{I}$-model or $-\infty$ for each $\tau\in \mathbb{R}$.

\begin{lemma}\label{lma:val}
Let $\psi_{\bullet}$ be the test curve induced by a filtration $\mathscr{F}^{\bullet}$ on $R(X,L)$.
Let $v$ be a real valuation of $\mathbb{C}(X)$. Then
\[
v(\psi_{\tau})=\inf_{k\in \mathbb{Z}_{>0}} k^{-1} \inf\left\{\,v(s):s\in \mathscr{F}^{k\tau}H^0(X,L^k)\,\right\}\,.
\]
\end{lemma}
\begin{proof}
For $k\in \mathbb{Z}_{>0}$, let
\[
F_k:=\sups\left\{\,\log|s|_{h^k}^2: s\in \mathscr{F}^{k\tau}H^0(X,L^k),\sup_X|s|_{h^k}\leq 1\,\right\}\,.
\]
For $k,m\in \mathbb{Z}_{>0}$,
\[
F_{k+m}\geq F_k+F_m\,.
\]
So by Fekete's lemma, $\psi_{\tau}$ is the usc regularization of the increasing limit $2^{-k}F_{2^k}$. We conclude by the monotonicity and the upper semi-continuity of Lelong numbers.
\end{proof}

\begin{lemma}\label{lma:volfiltc}
Let $\psi_{\bullet}$ be the test curve induced by a $\mathbb{Z}$-filtration $\mathscr{F}^{\bullet}$ on $R(X,L)$. Then
\begin{equation}\label{eq:t8}
\int_X \omega_{\psi_{\tau}}^n\geq\lim_{k\to\infty} \frac{n!}{k^n}\dim \mathscr{F}^{k\tau}H^0(X,L^k)\,.
\end{equation}
Equality holds if $\mathscr{F}^{\bullet}$ is finitely generated, $\tau<\tau^+$.
\end{lemma}
Note that the limit on the right-hand side exists by \cite{LM09}.
\begin{proof}
By \cite[Theorem~1.1]{DX22},
\[
\int_X \omega_{\psi_{\tau}}^n=\lim_{k\to\infty}\frac{n!}{k^n}h^0(X,L^k\otimes \mathscr{I}(k\psi_{\tau}))\,.
\]
Each element in $\mathscr{F}^{k\tau}H^0(X,L^k)$ is obviously square integrable with respect to $k\psi_{\tau}$, \eqref{eq:t8} follows.

Now assume that $\mathscr{F}^{\bullet}$ is finitely generated. Then it is the filtration induced by some test configuration $(\mathcal{X},\mathcal{L})$ of $(X,L)$ by \cite[Proposition~2.15]{BHJ17}. Without loss of generality, we may assume that $\tau^+=0$ for the test curve $\psi_{\bullet}$. Then by \cite[Section~5]{BHJ17}, the Duistermaat--Heckman measure of $(\mathcal{X},\mathcal{L})$ is given by
\[
\nu=-\frac{1}{V}\frac{\mathrm{d}}{\mathrm{d}\tau}\vol(R^{(\tau)})\,,
\]
where
\[
\vol(R^{(\tau)}):=\lim_{k\to\infty}\frac{n!}{k^n}\dim \mathscr{F}^{k\tau}H^0(X,L^k)\,.
\]
By \cite[Lemma~7.3]{BHJ17}, the non-Archimedean Monge--Amp\`ere energy of $(\mathcal{X},\mathcal{L})$ is given by
\[
E^{\NA}(\mathcal{X},\mathcal{L})=\int_{-\infty}^{\infty}\tau\,\mathrm{d}\nu(\tau)=\int_{-\infty}^0 \left(\frac{1}{V}\vol R^{(\tau)}-1\right)\,\mathrm{d}\tau\,.
\]
On the other hand, by \cite[Theorem~1.1]{DX22}, 
\[
E^{\NA}(\mathcal{X},\mathcal{L})=\int_{-\infty}^0 \left(\frac{1}{V}\int_X \omega_{\psi_{\tau}}^n-1\right)\,\mathrm{d}\tau\,.
\]
Now by \eqref{eq:t8}, \cref{prop:tc} and \cite[Theorem~5.3]{BHJ17}, we conclude that equality holds in \eqref{eq:t8} when $\tau<\tau^+$.
\end{proof}

Let $(\mathcal{X},\mathcal{L})$ be a test configuration of $(X,L)$. 
It induces a filtration as follows: Take $r\in \mathbb{Z}_{>0}$ so that $\mathcal{L}^r$ is integral. Then $(\mathcal{X},\mathcal{L})$ induces a $\mathbb{Z}$-filtration of $R(X,rL)$ as follows: let $s\in H^0(X,rkL)$, then $s\in \mathscr{F}^{\lambda}H^0(X,rkL)$ if{f} $t^{-\lambda}s\in H^0(\mathcal{X},\mathcal{L}^{rk})$. Here we have abused the notation by writing $s$ for the equivariant extension of $s$ as well.
See \cite{BHJ17}. The weight of the $\mathbb{C}^*$-action on the central fibre of $\mathcal{L}^{rk}$ is given by
    \[
        w(rk)=-\int_{-\infty}^{\infty}\lambda \,\mathrm{d}\dim\mathscr{F}^{\lambda}H^0(X,L^{rk})\,.
    \]

\begin{example} \label{ex:blflag}
Let $I=I_0+I_1t+\cdots +I_{N-1}t^{N-1}+(t^N)$ be a flag ideal on $X\times \mathbb{P}^1$. Let $\mathcal{X}=\Bl_I X\times \mathbb{P}^1$. Denote by $\Pi:\mathcal{X}\rightarrow X\times\mathbb{P}^1$ the natural morphism. Let $E$ be the exceptional divisor. Let $p_1:X\times\mathbb{P}^1 \rightarrow X$ be the natural projection.
Assume that $\mathcal{L}:=\Pi^*p_1^*L\otimes \mathcal{O}_{\mathcal{X}}(-E)$ is $\pi$-semiample.

Write 
\[
I^k=\sum_{j=0}^{Nk-1}J_{k,j}t^j+(t^{Nk})\,.
\]
Then
\[
J_{k,j}=\sum_{\alpha\in \mathbb{N}^{N},|\alpha|=k,|\alpha|'=j} I_{\bullet}^{\alpha}\,.
\]
Here $|\alpha|':=\sum_i i\alpha_i$.
Set $J_{k,kN}=\mathcal{O}_X$.

Let $\mathscr{F}^{\bullet}$ be the filtration on $R(X,L)$ induced by $(\mathcal{X},\mathcal{L})$. 
 Let $\psi_{\bullet}$ be the corresponding test curve.

We claim that
\[
\psi_{\tau}^{\An}(v)=-\min_{\alpha\in \mathbb{Q}_{\geq 0}^{N},|\alpha|=1,|\alpha|'=-\tau}\sum_{i}\alpha_i v(I_i)\,.
\]
In particular,
\[
\psi^{\An}_{0}(v)=-v(I_0)\,.
\]
\end{example}
\begin{proof}
Let $\lambda \in \mathbb{Z}$ and $s\in H^0(X,L^{k})$, then $s\in \mathscr{F}^{\lambda}H^0(X,L^{k})$ if{f} $t^{-\lambda}s$ extends to a section of $\mathcal{L}^{k}$ if{f}
\[
t^{-\lambda}s \in H^0(X\times \mathbb{C}, L^{k}\otimes I^{\otimes k})
\]
if{f}  $s\in J_{k,-\lambda}$. Hence we have
\[
v(\psi_{\tau})=\inf_k \frac{1}{k} \inf \{v(s):s\in J_{k,-\ceil{k\tau}}\}\,.
\]
Observe that
\[
\inf \{v(s):s\in J_{k,-\ceil{k\tau}}\}=\min_{\alpha\in \mathbb{N}^{N},|\alpha|=k,|\alpha|'=-\ceil{k\tau}} \sum_{i}\alpha_i v(I_i)\,.
\]
So
\[
v(\psi_{\tau})\geq \min_{\alpha\in \mathbb{Q}_{\geq 0}^{N},|\alpha|=1,|\alpha|'=-\tau}\sum_{i}\alpha_i v(I_i)\,.
\]
On the other hand, observe that the minimizer is indeed rational when $\tau$ is rational,  so the reverse inequality also holds.
\end{proof}
Observe that when $\tau<0$, $\psi_{\tau}$ has quasi-analytic singularities (\cref{def:qasing}).
\begin{example}
Let $v=c\ord_F$ be a divisorial valuation of $\mathbb{C}(X)$, where $c\in \mathbb{Q}_{>0}$, $F$ is a prime divisor over $X$. Then $v$ induces a filtration $\mathscr{F}_v^{\bullet}$ on $R(X,L)$:
\[
\mathscr{F}_v^{\lambda}H^0(X,L^k)=
\left\{
\begin{aligned}
H^0(X,kL-\lambda cF)&\,,\quad \lambda\geq 0\,,\\
H^0(X,L^k)&\,,\quad \lambda<0\,.
\end{aligned}
\right.
\]
Here we have omitted the pull-back of $L$ to a model.
\end{example}

\begin{example}\label{ex:def}
Let $\psi\in \PSH(X,\omega)$ be a potential with analytic singularities. Let $\pi:Y\rightarrow X$ be a log resolution of the singularities of $\psi$. Assume that $\pi^*L-\Div_Y\psi$ is semi-ample. Then $\psi$ induces a test configuration of $(Y,\pi^*L)$ by deformation to the normal cone with respect to $\Div_Y\psi$. Then a section $s\in H^0(X,L^k)=H^0(Y,\pi^*L^k)$ is in $\mathscr{F}^{\lambda}$ if{f} $t^{-\lambda}s$ extends to the central fibre, that is, $s\in \mathcal{O}_Y(-(k+\lambda) \Div_Y \psi)$. Hence
\[
\mathscr{F}_{\psi}^{\lambda}H^0(X,L^k)=
\left\{
\begin{aligned}
H^0(Y,k\pi^*L-(\lambda+k) \Div_Y \psi)&\,,\quad \lambda\leq 0\,,\\
H^0(X,L^k)&\,,\quad \lambda>0\,.
\end{aligned}
\right.
\]
The test curve $\psi_{\bullet}$ defined in \cref{subsec:defn} is induced by this filtration.
\end{example}
\begin{example}\label{ex:ext}
Let $\psi\in \PSH(X,\omega)$ be a potential with analytic singularities. Let $\pi:Y\rightarrow X$ be a log resolution of the singularities of $\psi$. The deformation to the normal cone defined in \cref{ex:def} can be extended as follows:
\[
\mathscr{F}_{\psi^+}^{\lambda}H^0(X,L^k)=
\left\{
\begin{aligned}
H^0(X,k\pi^*L-\lambda \Div_Y \psi)&\,,\quad \lambda\geq 0\,,\\
H^0(X,L^k)&\,,\quad \lambda<0\,.
\end{aligned}
\right.
\]
The test curve $\psi^+_{\bullet}$ defined in \cref{subsec:extdef} is induced by this filtration. 
\end{example}

\subsection{The Phong--Sturm geodesic ray}
Let $X$ be a compact Kähler manifold of dimension $n$. Let $\omega$ be a real smooth \emph{semi-positive} $(1,1)$-form on $X$.
Let $(\mathcal{X},\mathcal{L})$ be a semi-ample test configuration of $(X,L)$. Fix a $S^1$-invariant smooth metric $\Phi$ on $\mathcal{L}$ with $c_1(\mathcal{L},\Phi)=\Omega$, we may assume that $\Omega|_{X\times S^1}$ is the pull-back of $\omega$. Let $\pi:\mathcal{X}\rightarrow \mathbb{C}$ be the natural map.
 Let $\mathcal{X}^{\circ}:=\pi^{-1}(\Delta)$, where $\Delta=\{z\in \mathbb{C}:|z|<1\}$.
Consider the homogeneous Monge--Ampère equation
\begin{equation}\label{eq:HCMA}
\left\{
\begin{aligned}
(\Omega+\ddc \Psi)^{n+1}&=0\quad \text{on } \mathcal{X}^{\circ}\,,\\
\left.\Psi\right|_{X\times S^1}&=0\,.
\end{aligned}
\right.
\end{equation}
By \cite{CTW18}, there is a unique bounded solution  to \eqref{eq:HCMA} and the solution is $C^{1,1}$ outside the central fibre.

Let $\ell$ be the geodesic ray in $\mathcal{E}^1(X,\omega)$ corresponding to $\Psi$, then $\ell$ is known as the \emph{Phong--Sturm geodesic ray} induced by $(\mathcal{X},\mathcal{L})$. This construction was first studied in \cite{PS07} and \cite{PS10}.

\begin{theorem}[{\cite[Theorem~9.2]{RWN14}}]
	Let $(\mathcal{X},\mathcal{L})$ be a semi-ample test configuration of $(X,L)$. Let $\ell$ be the Phong--Sturm geodesic ray induced by $(\mathcal{X},\mathcal{L})$. Let $\mathscr{F}^{\bullet}$ be the filtration induced by $(\mathcal{X},\mathcal{L})$. Let $\psi_{\bullet}$ be the test curve induced by $\mathscr{F}^{\bullet}$ as in \eqref{eq:testcurvfromfilt}. Then $\check{\psi}=\ell$.
\end{theorem}

\subsection{Non-Archimedean analogue of Ross--Witt Nystr\"om correspondence}
Assume that $L$ is ample. 
\begin{definition}
 A function $\psi:X^{\An}\rightarrow [-\infty,\infty)$ is called a \emph{good potential} if there exists $\varphi\in \PSH(X,\omega)$ such that $\psi=\varphi^{\An}$.
 
 The set of good potential is denoted as $\PSH^{\NA}_{\mathrm{g}}(X,\omega)$.
\end{definition}
See \eqref{eq:psian} for the definition of $\varphi^{\An}$.
\begin{proposition}
 The map $\psi\mapsto \psi^{\An}$ is a bijection from $\PSH^{\Mdl}_{\mathscr{I}}(X,\omega)$ to $\PSH^{\NA}_{\mathrm{g}}(X,\omega)$.
\end{proposition}
This is obvious by definition.

\begin{definition}
 A test curve $\psi\in \TC^{\infty}(X,\omega)$ is \emph{piecewise linear} if $\psi^{\An}$ is piecewise linear with finitely many breaking points (i.e. non-differentiable points).
\end{definition}

\begin{definition}\label{def:NATC}
 A \emph{non-Archimedean test curve} is a map $\psi:(-\infty,\tau^+)\rightarrow \PSH^{\NA}_{\mathrm{g}}(X,\omega)$ for some $\tau^+\in \mathbb{R}$, such that
 \begin{enumerate}
    \item $\psi$ is concave.
    \item $\lim_{\tau\to-\infty}\psi_{\tau}=0$ in $L^1$.
\end{enumerate}
We define $\tau^-$ as in the Archimedean case.

The non-Archimedean test curve $\psi_{\bullet}$ is \emph{of finite energy} if
\begin{equation}\label{eq:defEpsib}
\mathbf{E}(\psi_{\bullet}):=\tau^++\frac{1}{V}\int_{-\infty}^{\tau^+} \left(\int_X \omega_{\varphi_{\tau}}^n-\int_X\omega^n\right)\,\mathrm{d}\tau>-\infty\,,
\end{equation}
where $\varphi_{\tau}$ is the $\mathscr{I}$-model potential in $\PSH(X,\omega)$ with $\varphi_{\tau}^{\An}=\psi_{\tau}$.

The set of non-Archimedean test-curves of finite energy is denoted by $\TC^{1,\NA}(X,\omega)$.
\end{definition}
\begin{proposition}
 The map $\TC^1_{\mathscr{I}}(X,\omega) \to \TC^{1,\NA}(X,\omega)$ defined by 
 $\psi_{\bullet}\mapsto (\psi^{\An}_{\tau})_{\tau<\tau^+}$
 is a bijection.
\end{proposition}
This again is immediate by definition.
\begin{remark}
    In \cref{def:NATC}, we deliberately define $\psi_{\tau}$ only for $\tau<\tau^+$. This is because it is \emph{not} always true that for an Archimedean test curve $\psi_{\bullet}$, 
    \[
    \psi_{\tau^+}^{\An}=\lim_{\tau\to \tau^+-}\psi_{\tau}^{\An}\,.
    \]
\end{remark}
\begin{theorem}[{\cite[Proposition~3.13]{DX22}}]\label{thm:legna}
    The map $\check{}:\TC^{1,\NA}(X,\omega)\to \mathcal{E}^{1,\NA}$ given by
    \[
    \psi_{\bullet}^{\An} \mapsto \sup_{\tau<\tau^+}(\psi_{\tau}^{\An}+\tau)
    \]
    is a bijection. Moreover, when $\psi_{\bullet}\in \TC^1(X,\omega)$, 
    \[
        \left(\psi_{\bullet}^{\An}\right)^{\check{}}=\left(\check{\psi}_{\bullet}\right)^{\NA}\,,
    \]
    namely, the following diagram commutes:
    \[
    \begin{tikzcd}
        \TC^{1} \arrow[r, "\An"] \arrow[d, "\check{}"]
        & \TC^{1,\NA} \arrow[d, "\check{}"] \\ \mathcal{R}^{1} \arrow[r, "\NA"] &\mathcal{E}^{1,\NA} 
    \end{tikzcd}
    \]
\end{theorem}

\begin{remark}
Ideally when we are considering only maximal geodesic rays, it should be possible to carry out the computations in \cref{sec:for} purely in terms of non-Archimedean test curves, without referring to the machinery of test curves, filtrations and test configurations. However, the difficulty is that we do not have a good understanding of the following non-Archimedean Monge--Ampère measure
\[
\MA\left(\sup_{\tau<\tau^+}(\psi_{\tau}^{\An}+\tau)\right)\,.
\]
It is highly desirable to have a description of this measure in terms of certain real Monge--Ampère measures on some dual complexes. In the non-trivially valued case, a partial result is derived by Vilsmeier (\cite{Vil20}).
\end{remark}

\section{Intersection theory of b-divisors}\label{sec:bdiv}
In this section, we apply the intersection theory of Shokurov's b-divisors to the study of singularities of psh functions. Due to the technical assumptions in \cite{DF20} and \cite{DF20a}, we can not apply Dang--Favre's intersection theory directly. 
Although it seems possible to remove the technical assumptions in Dang--Favre's theory, we do not pursue this most general theory here.\footnote{When the current paper was written, the second version of \cite{DF20} was not available yet, where Dang--Favre developed the general intersection theory of nef b-divisors. Our definition of volumes is essentially the same as the intersection number defined using \cite{DF20}.}

References to this section are \cite{DF20}, \cite{DF20a}, \cite{BFJ09}, \cite{BDPP13}, \cite{KK14}.

Let $X$ be a projective manifold of dimension $n$. 

\subsection{b-divisors}
Recall that the Riemann--Zariski space of $X$ is the locally ringed space defined by
\[
\mathfrak{X}:=\varprojlim_{Y} Y\,,
\]
where $Y$ runs over all birational models of $X$. Here the projective limit is taken in the category of locally ringed spaces. For valuative interpretation of $\mathfrak{X}$, see \cite{Tem11}. We do not make use of the theory of Riemann--Zariski spaces in an essential way in this paper. Instead, we give an \emph{ad hoc} treatment of divisors on $\mathfrak{X}$.

\begin{definition}
 By a \emph{Weil divisor on $\mathfrak{X}$} or a \emph{Weil b-divisor} on $X$, we mean an element in
 \[
 \bWeil(X):=\varprojlim_{Y} \Weil(Y)\,,
 \]
 where $Y$ runs over all (smooth) birational models of $X$ and $\Weil(Y)$ is the set of numerical classes of $\mathbb{R}$-divisors on $Y$.
 
 By a \emph{Cartier divisor on $\mathfrak{X}$} or a \emph{Cartier b-divisor} on $X$, we mean an element in
 \[
 \bCart(X):=\varinjlim_{Y} \Weil(Y)\,,
 \]
 where $Y$ runs over all (smooth) birational models of $X$.
 
 Both the limit and the colimit are taken in the category of topological vector spaces.
\end{definition}
There is a natural continuous injection $\bCart(X) \hookrightarrow \bWeil(X)$.

\subsection{Differentiability of the volume}\label{subsec:diff}
General references of results in this section are \cite{BFJ09}, \cite{DP04}.

Let $X$ be a compact Kähler manifold of dimension $n$.
Let $L$ be a big line bundle on $X$. Recall that the volume of $L$ is defined as
\[
\vol(L):=\lim_{k\to\infty}\frac{n!}{k^n}h^0(X,L^k)\,.
\]
More generally, by requiring
\[
\vol(L^k)=k^n\vol(L)\,,
\]
we extend the definition of volume to all big $\mathbb{Q}$-line bundles. By continuity, this definition further extends to all pseudo-effective $\mathbb{R}$-line bundles.

When $L$ is a nef $\mathbb{R}$-line bundle, we have
\begin{equation}
\vol(L)=(L^n)\,.
\end{equation}
Recall the following basic fact,
\begin{theorem}[{\cite{BFJ09}}]\label{thm:diffv}
The volume function $\vol$ is continuously differentiable in the big cone. Moreover, let $L$ be a big and nef $\mathbb{R}$-line bundle, let $L'$ be a line bundle, then
\begin{equation}
\left.\frac{\mathrm{d}}{\mathrm{d}\epsilon}\right|_{\epsilon=0}\vol(L+\epsilon L')=n \left(L^{n-1}\cdot L'\right)\,.
\end{equation}
\end{theorem}

Now assume that $L$ is  big and semi-ample.
Fix a smooth semi-positive real $(1,1)$-form $\omega\in c_1(L)$. Let $\psi\in \PSH(X,\omega)$ be a potential with quasi-analytic singularities along a snc $\mathbb{R}$-divisor $\Div_X\psi$. Assume that $\psi$ has positive mass. Recall that by \cref{lma:diffnab}, $L-\Div_X\psi$ is nef and big. Let $L'$ be an $\mathbb{R}$-line bundle on $X$.
Now we define 
\begin{equation}\label{eq:voldiff1}
    D_L(\psi,L')=\left.\frac{\mathrm{d}}{\mathrm{d}\epsilon}\right|_{\epsilon=0}\vol(L-\Div_X\psi +\epsilon L')=n\left((L-\Div_X\psi)^{n-1}\cdot L'\right)\,.
\end{equation}
When $\psi\in \PSH(X,\omega)$ has positive mass and there exists a birational model $\pi:Y\rightarrow X$, $\psi$ has quasi-analytic singularities along a snc $\mathbb{R}$-divisor $\Div_Y\psi$, let $L'$ be an $\mathbb{R}$-line bundle on $Y$,
we define
\begin{equation}\label{eq:voldiff2}
    D_L(\psi,L'):=D_{\pi^*L}(\pi^*\psi,L')\,.
\end{equation}
We formally set $D_L(-\infty,L')=0$.

\subsection{Singularity divisors}
Let $L$ be a semi-ample line bundle on $X$. Let $h$ be a non-negatively curved metric on $L$. Let $\omega=c_1(L,h)$.
\begin{definition}
Let $\psi\in \PSH(X,\omega)$. We define the \emph{singularity divisor} of $\psi$ as a Weil b-divisor $\Div_{\mathfrak{X}} \psi\in \bWeil(X)$:
 \[
 \left(\Div_{\mathfrak{X}} \psi\right)_Y=\Div_Y\psi\,.
 \]
 Here we have abused the notation by writing $\Div_Y\psi$ for the numerical class of the corresponding divisor, which makes sense as explained in \cref{rmk:divNS}.
\end{definition}

We set
\[
\vol(L-\Div_{\mathfrak{X}}\psi):=\lim_Y \vol(\pi^*L-\Div_Y\psi)\,,
\]
where $\pi:Y\rightarrow X$ runs over all birational model of $X$. The net is decreasing, hence the limit is well-defined.

\begin{theorem}\label{thm:volIm}
Assume that $\psi$ is $\mathscr{I}$-model and of positive mass, then
\begin{equation}
\int_X \omega_{\psi}^n=\vol\left(L-\Div_{\mathfrak{X}}\psi \right)\,.
\end{equation}
\end{theorem}
\begin{proof}
Let $\psi^j$ be a quasi-equisingular approximation to $\psi$. By \cite[Theorem~1.4]{DX22}, $\int_X \omega_{\psi^j}^n
\to \int_X \omega_{\psi}^n$. Similarly, the right-hand side converges along $\psi^j$ as follows from \cite[Proof of Theorem~6(3)]{DF20}. To be more precise, it suffices to prove that for any $\epsilon>0$, any model $\pi:Y\rightarrow X$, we can find $j_0>0$, such that for $j\geq j_0$,
\[
\vol\left(L-\Div_{\mathfrak{X}}\psi \right)\leq \vol\left(L-\Div_{\mathfrak{X}}\psi^j\right)\leq \vol\left(\pi^*L-\Div_{Y}\psi \right)+\epsilon\,.
\]
The first inequality is trivial. For the second inequality, observe that by \cref{lma:qesana}, $\Div_Y\psi^j \to \Div_Y\psi$. Fix some $C>0$, depending on $\pi$,
 we may take  $j_0$ large enough, so that when $j\geq j_0$,
\[
\pi^*L-\Div_Y\psi^j\leq  \pi^*L-\Div_Y\psi+C^{-1}\epsilon\pi^*\omega\,.
\]
Then it follows that
\[
\vol\left(\pi^*L-\Div_Y\psi^j\right)\leq \vol\left( \pi^*L-\Div_Y\psi \right)+\epsilon\,.
\]
Hence 
\[
\vol\left(L-\Div_{\mathfrak{X}}\psi^j\right)\leq \vol\left( \pi^*L-\Div_Y\psi \right)+\epsilon\,.
\]
\end{proof}
In particular, this gives an additional characterization of $\mathscr{I}$-model potentials.
\begin{corollary}\label{cor:charImdlint}
Let $\psi\in \PSH^{\Mdl}(X,\omega)$ be a model potential with positive mass. Then $\psi$ is $\mathscr{I}$-model if{f}
\[
    \int_X \omega_{\psi}^n=\vol\left(L-\Div_{\mathfrak{X}}\psi \right)\,.
\]
\end{corollary}

\begin{remark}
As the techniques of \cite{DX22} have been extended to pseudo-effective line bundles in \cite{DX21}, this corollary and its proof actually work in the setting of big line bundles. In terms of \cite{DF20}, our proof also shows that $L-\Div_{\mathfrak{X}}\psi$ is nef. A special case of this result is also discovered in \cite{BBGHdJ}.
\end{remark}

\section{Radial functionals in terms of Legendre transforms}\label{sec:for}
In this section, let $X$ be a compact Kähler manifold of dimension $n$. 
Let $L$ be a big and semi-ample line bundle on $X$. Let $h$ be a smooth non-negatively curved metric on $L$. Let $\omega=c_1(L,h)$.

From \cref{subsec:MAE} on, we assume that $L$ is an ample line bundle and $h$ is strictly positively curved.

In this section, we study several functionals on the space of geodesic rays and express them in terms of test curves.

\subsection{Functionals on the space of test curves}
Let $\psi_{\bullet}\in \TC^1(X,\omega)$. Recall that $\tau^+:=\inf \{\tau\in \mathbb{R}: \psi_{\tau}=-\infty\}$.

We have already defined the Monge--Ampère energy $\mathbf{E}(\psi_{\bullet})$ in \eqref{eq:defE}.
For any real smooth $(1,1)$-form $\alpha$ on $X$, define the $\alpha$-energy of $\psi_{\bullet}$ as
\begin{equation}
\mathbf{E}^{\alpha}(\psi_{\bullet}):=\tau^+\frac{1}{V}\int_X \alpha\wedge\omega^{n-1}+\frac{1}{V}\int_{-\infty}^{\tau^+} \left(\int_X \alpha\wedge \omega_{\psi_{\tau}}^{n-1}-\int_X\alpha\wedge \omega^{n-1}\right)\,\mathrm{d}\tau\,.
\end{equation}
The \emph{Ricci energy} of $\psi_{\bullet}$ is defined as
\begin{equation}
\mathbf{E}_R(\psi_{\bullet}):=-n\tau^+\frac{1}{V}\int_X \Ric\omega'\wedge\omega^{n-1}-\frac{n}{V}\int_{-\infty}^{\tau^+} \left(\int_X \Ric\omega'\wedge \omega_{\psi_{\tau}}^{n-1}-\int_X\Ric\omega'\wedge \omega^{n-1}\right)\,\mathrm{d}\tau\,,
\end{equation}
where $\omega'$ denotes a Kähler form on $X$.

The \emph{$\tilde{J}$-functional} of $\psi_{\bullet}$ is defined as
\begin{equation}
\tilde{\mathbf{J}}(\psi_{\bullet})=n\mathbf{E}^{\omega}(\psi_{\bullet})-n \mathbf{E}(\psi_{\bullet})=\frac{n}{V}\int_{-\infty}^{\infty} \left(\int_X \omega\wedge \omega_{\psi_{\tau}}^{n-1}-\int_X\omega_{\psi_{\tau}}^n\right)\,\mathrm{d}\tau\,.
\end{equation}

\begin{remark}
    It is interesting to observe that $\mathbf{E}^{\alpha}(\psi_{\bullet})$ depends only on the cohomology class of $\alpha$.
\end{remark}

Assume that $\psi_{\bullet}$ is $\mathscr{I}$-model.
The \emph{non-Archimedean $\mathscr{L}_k$-functional} of $\psi_{\bullet}$ is defined as
\begin{equation}\label{eq:defL}
\mathscr{L}_k^{\NA}(\psi_{\bullet}):=\frac{1}{V}\int_{-\infty}^{\infty} \tau \,\mathrm{d}h^0(X,K_X\otimes L^k\otimes \mathscr{I}(k\psi_{\tau}))\,.
\end{equation}

Assume that $\psi_{\bullet}$ is $\mathscr{I}$-model, the \emph{entropy} of $\psi_{\bullet}$ is defined as
\begin{equation}
\Ent(\psi_{\bullet}):=\int_{-\infty}^{\infty}\Ent([\psi_{\tau}])\,\mathrm{d}\tau\,.
\end{equation}
Recall that $\Ent[\bullet]$ is defined in \cref{def:naentgeneral}.

\begin{definition}\label{def:anatc}
Let $\psi_{\bullet}\in \TC^{\infty}(X,\omega)$. 
We say $\psi_{\bullet}$ is \emph{analytic} if  $\psi_{\tau}$ has quasi-analytic singularities for any $\tau<\tau^+$.

We say $\psi_{\bullet}$  is \emph{piecewise linear} if $\psi_{\bullet}^{\An}$ is piecewise linear with finitely many breaking points (non-differentiable points).
\end{definition}

We need the following observation.
\begin{lemma}\label{lma:exana}
The test curves in \cref{ex:blflag} are analytic and piecewise linear. 
\end{lemma}

\begin{corollary}\label{cor:tcana}
    The test curve induced by a test configuration is analytic and piecewise linear.
\end{corollary}
\begin{proof}
This follows from \cite[Proposition~3.10]{Oda13} and \cref{lma:exana}.
\end{proof}
\begin{remark}
The statement of \cite[Proposition~3.10]{Oda13} needs to be corrected as follows: $\mathcal{L}^r(-E)=f^*\mathcal{M}+c\mathcal{B}_0$ for some constant $c\in \mathbb{Q}$. The mistake in the proof is on the fourth line, where we need to make sure that the isomorphism between $h^*\mathcal{M}^{s}$ and $\mathcal{L}^{r}$ extends to the generic point of the central fibre.
\end{remark}

We observe the following obvious lemma.
\begin{lemma}
 Let $\psi_{\bullet}$ be an analytic test curve.  Then
\begin{equation}\label{eq:EntNAtc}
    \Ent(\psi_{\bullet})=\frac{1}{V}\int_{-\infty}^{\infty}D_L(\psi_{\tau},K_{Y_{\tau}/X})\,\mathrm{d}\tau+\frac{1}{V}\int_{-\infty}^{\infty} D_L(\psi_{\tau},\Redu\Div_Y\psi_{\tau})\,\mathrm{d}\tau\,,
\end{equation}
where $\pi_{\tau}:Y_{\tau}\rightarrow X$ is a log resolution of $\psi_{\tau}$.
\end{lemma}
See \cref{subsec:diff} for the definition of $D_L$.

\subsection{Monge--Ampère energy}\label{subsec:MAE}
From this section on, we assume that $L$ is ample and $h$ is strictly positively curved, so that $\omega$ is a K\"ahler form.

\begin{theorem}[{\cite[Theorem~3.7]{DX22}}]\label{thm:LegE}
Let $\ell\in \mathcal{R}^1$. Then
\begin{equation}
\mathbf{E}(\ell)=\mathbf{E}(\hat{\ell})\,.
\end{equation}
\end{theorem}

Recall that the right-hand side is defined in \eqref{eq:defE}.
\subsection{Non-archimedean \texorpdfstring{$\mathcal{L}$}{L}-functionals}

\begin{theorem}[{\cite[Theorem~1.1]{DX22}}]\label{thm:LegLk}
Let $\ell\in \mathcal{E}^{1,\NA}$.
For each $k\in \mathbb{Z}_{>0}$,
\begin{equation}
\mathscr{L}_k^{\NA}(\ell)=\mathscr{L}_k^{\NA}(\hat{\ell})\,.
\end{equation}
\end{theorem}

The right-hand side is defined in \eqref{eq:defL} and the left-hand side is defined in \eqref{eq:defLkr}.
\subsection{\texorpdfstring{$\alpha$}{\alpha}-energy}
Let $\alpha$ be a smooth real $(1,1)$-form on $X$. 

\begin{lemma}\label{lma:derMPhi}
Let $\varphi,\psi\in \mathcal{E}^{\infty}$, then
\[
\left.\frac{\mathrm{d}}{\mathrm{d}s}\right|_{s=0}E^{\alpha}(s\psi+(1-s)\varphi)=\frac{1}{V}\int_X (\psi-\varphi)\,\alpha\wedge \omega_{\varphi}^{n-1}\,.
\]
\end{lemma}
\begin{proof}
This result is well-known when $\psi$ and $\varphi$ are smooth. In general, it follows from a direct computation using integration by parts (\cite{Xia19b}, \cite{Lu21}).
\end{proof}

\begin{theorem}\label{thm:LegEalpha}
Let $\ell\in \mathcal{E}^{1,\NA}$ or $\ell\in \mathcal{R}^{\infty}$. Then
\begin{equation}
\mathbf{E}^{\alpha}(\ell)= \mathbf{E}^{\alpha}(\hat{\ell})\,.
\end{equation}
\end{theorem}

The strategy of the proof first appeared in \cite{RWN14}.

\begin{proof}
Without loss of generality, we may assume that $\alpha$ is a Kähler form and $\sup_X \ell_1=0$.

We first assume that $\ell\in \mathcal{R}^{\infty}$.
We fix a few notations.
Let $\psi_{\bullet}$ be the Legendre transform of $\ell$.
Now for each $N\in \mathbb{N}$, $M\in \mathbb{Z}$, $t\geq 0$, we introduce
\[
\check{\psi}_t^{N,M}:=\max_{\substack{k\in \mathbb{Z}\\ k\leq M}} (\psi_{k2^{-N}}+tk2^{-N})\,.
\]
Let
\[
U_t^{N,M}:=\left\{\,x\in X:\check{\psi}_t^{N,M+1}(x)>\check{\psi}_t^{N,M}(x)\,\right\}\,.
\]
Observe that on $U_t^{N,M}$,
\begin{equation}\label{eq:checkpsivar}
\check{\psi}_t^{N,M+1}=\psi_{(M+1)2^{-N}}+t(M+1)2^{-N}\,,\quad \check{\psi}_t^{N,M}=\psi_{M2^{-N}}+tM2^{-N}\,.    
\end{equation}
By \cref{lma:derMPhi},
\[
E^{\alpha}(\check{\psi}_t^{N,M+1})-E^{\alpha}(\check{\psi}_t^{N,M})=\frac{1}{V}\bigintss_{\,\,0}^1 \bigintss_X \left(\check{\psi}_t^{N,M+1}-\check{\psi}_t^{N,M}\right)\alpha\wedge \omega_{s\check{\psi}_t^{N,M+1}+(1-s)\check{\psi}_t^{N,M}}^{n-1} \,\mathrm{d}s\,.
\]
By the comparison principle (\cite[Proposition~3.5]{DDNL18mono}),
\[
\begin{aligned}
\int_{U_t^{N,M}} (\check{\psi}_t^{N,M+1}-\check{\psi}_t^{N,M})\alpha\wedge \omega_{\check{\psi}_t^{N,M+1}}^{n-1}\leq &
\int_X (\check{\psi}_t^{N,M+1}-\check{\psi}_t^{N,M})\alpha\wedge \omega_{s\check{\psi}_t^{N,M+1}+(1-s)\check{\psi}_t^{N,M}}^{n-1}\\
\leq & \int_{U_t^{N,M}} (\check{\psi}_t^{N,M+1}-\check{\psi}_t^{N,M})\alpha\wedge \omega_{\check{\psi}_t^{N,M}}^{n-1}\,.
\end{aligned}
\]
We first deal with the upper bound,
\[
E^{\alpha}(\check{\psi}_t^{N,M+1})-E^{\alpha}(\check{\psi}_t^{N,M})
\leq 2^{-N}V^{-1}t\int_{U_t^{N,M}} \alpha\wedge \omega_{\psi_{M 2^{-N}}}^{n-1}\,.
\]
Set $\tau^-:=\inf_X \ell_1$. 
Take the sum with respect to $M$ from $[\tau^-]2^N$ to $-1$, we get
\[
\begin{aligned}
-t[\tau^-]V^{-1}\int_X \omega^{n-1}\wedge \alpha+E^{\alpha}(\check{\psi}_t^{N,0})\leq & 2^{-N}V^{-1}t\sum_{M=[\tau^-]2^N}^{-1}\int_{U_t^{N,M}} \alpha\wedge \omega_{\psi_{M 2^{-N}}}^{n-1}\\
\leq & 2^{-N}V^{-1}t\sum_{M=[\tau^-]2^N}^{-1}\int_X \alpha\wedge \omega_{\psi_{M 2^{-N}}}^{n-1}\,.
\end{aligned}
\]
Let $N\to\infty$ and then $t\to\infty$, we get
\[
\mathbf{E}^{\alpha}(\ell)\leq \frac{1}{V}\int_{[\tau^-]}^0 \left(\int_X \alpha\wedge \omega_{\psi_{\tau}}^{n-1}-\int_X \alpha\wedge \omega^{n-1}\right)\,\mathrm{d}\tau\,.
\]
Now we deal with the lower bound part. We have
\[
\begin{split}
E^{\alpha}(\check{\psi}_t^{N,M+1})-E^{\alpha}(\check{\psi}_t^{N,M})
\geq & 2^{-N}V^{-1}t\int_{U_t^{N,M}} \alpha\wedge \omega_{\psi_{(M+1) 2^{-N}}}^{n-1}\\
&+V^{-1}\int_{U_t^{N,M}}(\psi_{(M+1)2^{-N}}-\psi_{M 2^{-N}}) \alpha\wedge \omega_{\psi_{(M+1) 2^{-N}}}^{n-1}\,.
\end{split}
\]
Taking summation with respect to $M$ from $[\tau^-]2^N$ to $-1$, we get
\[
\begin{aligned}
 \frac{1}{t}E^{\alpha}(\check{\psi}^{N,0}_t)\geq & 2^{-N}V^{-1}\sum_{M=[\tau^-]2^N}^{-1}\int_{U_t^{N,M}} \alpha\wedge \omega_{\psi_{(M+1) 2^{-N}}}^{n-1}\\
&+V^{-1}t^{-1}\sum_{M=[\tau^-]2^N}^{-1}\int_{U_t^{N,M}}(\psi_{(M+1)2^{-N}}-\psi_{M 2^{-N}}) \alpha\wedge \omega_{\psi_{(M+1) 2^{-N}}}^{n-1}\\
&+[\tau^-]V^{-1}\int_X \omega^{n-1}\wedge \alpha\,.
\end{aligned}
\]
Note that as $t\to\infty$, $\mathds{1}_{U_t^{N,M}}\to 1$ outside a pluripolar set if $M<-1$. Hence
\[
\begin{aligned}
\varliminf_{t\to\infty}\frac{1}{t}E^{\alpha}(\check{\psi}^{N,0}_t)\geq  & 2^{-N}V^{-1}\sum_{M=[\tau^-]2^N}^{-2}\int_X \alpha\wedge \omega_{\psi_{(M+1) 2^{-N}}}^{n-1}\\
&+\varliminf_{t\to\infty}(Vt)^{-1}\sum_{M=[\tau^-]2^N}^{-1}\int_{U_t^{N,M}}(\psi_{(M+1)2^{-N}}-\psi_{M 2^{-N}}) \alpha\wedge \omega_{\psi_{(M+1) 2^{-N}}}^{n-1}\\
&+[\tau^-]V^{-1}\int_X \alpha\wedge\omega^{n-1}\,.
\end{aligned}
\]
Observe that $\check{\psi}_t^{N,0}\leq \ell_t$, so
\[
\varliminf_{t\to\infty}\frac{1}{t}E^{\alpha}(\check{\psi}^{N,0}_t)\leq \lim_{t\to\infty}\frac{1}{t}E^{\alpha}(\ell_t)=\mathbf{E}^{\alpha}(\ell)\,.
\]
Observe that 
\[
U_{t}^{N,M}\setminus S \subseteq \left\{\,2^N(\psi_{(M+1)2^{-N}}-\psi_{M 2^{-N}})>-t\,\right\}\setminus S\,,
\]
where $S$ is the pluripolar set $\{\psi_{M 2^{-N}}=-\infty\}$.

Let
\[
F^N(t):=2^{-N}\sum_{M=[\tau^-]2^N}^{-1}\bigintss_{\left\{2^N(\psi_{(M+1)2^{-N}}-\psi_{M 2^{-N}})>-t\right\}}\alpha\wedge \omega_{\psi_{(M+1) 2^{-N}}}^{n-1}\,.
\]
Then
\[
\begin{split}
&\sum_{M=[\tau^-]2^N}^{-1}\int_{U_t^{N,M}}(\psi_{(M+1)2^{-N}}-\psi_{M 2^{-N}}) \alpha\wedge \omega_{\psi_{(M+1) 2^{-N}}}^{n-1}\\
\geq & 2^{-N}\sum_{M=[\tau^-]2^N}^{-1}\bigintss_{\left\{2^N(\psi_{(M+1)2^{-N}}-\psi_{M 2^{-N}})>-t\right\}}2^N(\psi_{(M+1)2^{-N}}-\psi_{M 2^{-N}})\,\alpha\wedge \omega_{\psi_{(M+1) 2^{-N}}}^{n-1}\\
=& -2^{-N}\sum_{M=[\tau^-]2^N}^{-1}\bigintss_0^t \,\mathrm{d}a \bigintss_{\left\{-a\geq 2^N(\psi_{(M+1)2^{-N}}-\psi_{M 2^{-N}})>-t\right\}}\alpha\wedge \omega_{\psi_{(M+1) 2^{-N}}}^{n-1}\\
=& -\int_0^t \left( F^N(t)-F^N(a) \right) \,\mathrm{d}a
\,.
\end{split}
\]
Observe that $F^N$ is bounded and increasing, so we conclude
\[
\varliminf_{t\to\infty}t^{-1}\sum_{M=[\tau^-]2^N}^{-1}\int_{U_t^{N,M}}(\psi_{(M+1)2^{-N}}-\psi_{M 2^{-N}}) \alpha\wedge \omega_{\psi_{(M+1) 2^{-N}}}^{n-1}\geq 0\,.
\]
We conclude 
\[
\mathbf{E}^{\alpha}(\ell)\geq \frac{1}{V}\int_{[\tau^-]}^0 \left(\int_X \alpha\wedge \omega_{\psi_{\tau}}^{n-1}-\int_X \alpha\wedge \omega^{n-1}\right)\,\mathrm{d}\tau\,.
\]

Now we deal with the case where $\ell\in \mathcal{E}^{1,\NA}$. It suffices to write $\ell$ as a decreasing limit of a sequence of Phong--Sturm geodesic rays $\ell^j\in \mathcal{R}^{\infty}$ as in \cite{BBJ15} and apply the monotone convergence theorem and \cite[(121)]{Li20}.
\end{proof}

\begin{corollary}\label{cor:Jtildeslope}
Let $\ell\in \mathcal{E}^{1,\NA}$, let $\psi=\hat{\ell}$, then
\[
\tilde{\mathbf{J}}(\ell)=\tilde{\mathbf{J}}(\psi_{\bullet})\,.
\]
\end{corollary}
\begin{proof}
This follows from \cref{thm:LegEalpha}, \cref{thm:LegE} and \eqref{eq:EomegaandE}.
\end{proof}

\begin{corollary}
Let $\ell^m\in \mathcal{E}^{1,\NA}$ ($m\in \mathbb{Z}_{>0}$) be a decreasing sequence of maximal geodesic rays. Let $\ell\in \mathcal{R}^1$ be its limit. Then $\mathbf{E}^{\alpha}(\ell^m)\to \mathbf{E}^{\alpha}(\ell)$ as $m\to\infty$.
\end{corollary}
This generalizes \cite[(121)]{Li20}. From our proof, it is easy to drop the condition that $\ell^m$ be decreasing when $\ell\in \mathcal{R}^{\infty}$.
\begin{proof}
Without loss of generality, we may assume that $\alpha$ is a K\"ahler form. We may and do assume that $\ell^m_0=0$, $\sup_X \ell^m_1=0$. 

Observe that $\ell$ is maximal by the completeness of $\mathcal{E}^{1,\NA}$ (see for example \cite[Theorem~1.2]{DX22}, \cite[Example~3.3]{Xia19b}).
By \cref{thm:LegEalpha}, it suffices to prove that
\[
\int_{-\infty}^0 \left(\int_X \alpha\wedge\omega_{\hat{\ell}^j_{\tau}}^{n-1}-\int_X \alpha\wedge\omega^{n-1}\right)\,\mathrm{d}\tau \to \int_{-\infty}^0 \left( \int_X \alpha\wedge\omega_{\hat{\ell}_{\tau}}^{n-1}-\int_X \alpha\wedge\omega^{n-1}\right)\,\mathrm{d}\tau\,.
\]
By monotone convergence theorem, it suffices to prove that for almost all $\tau< 0$,
\[
\int_X \alpha\wedge\omega_{\hat{\ell}^j_{\tau}}^{n-1}\to \int_X \alpha\wedge\omega_{\hat{\ell}_{\tau}}^{n-1}\,.
\]
It suffices to show that $\hat{\ell}_{\tau}$ is the $d_{\mathcal{S}}$-limit (\cite[Theorem~1.1]{DDNLmetric}) of $\hat{\ell}^j_{\tau}$ for almost all $\tau<0$. In turn, it suffices to show that $\int_X \omega_{\hat{\ell}^j_{\tau}}^n\to\int_X \omega_{\hat{\ell}_{\tau}}^n$ for almost all $\tau<0$. This follows from the continuity of $E^{\NA}$ along decreasing sequences and \cite[Theorem~1.1]{DX22}.
\end{proof}

\subsection{Entropy, dreamy quasi-psh function}\label{subsec:dreamy}
Results in this section are special cases of the results of \cref{sec:vBer}, we will be sketchy here.

Let $\psi\in \PSH(X,\omega)$ be a dreamy (\cref{def:dre}) potential with analytic singularities. Let $\pi:Y\rightarrow X$ be a log resolution of $\psi$, which is a composition of blowing-ups with smooth centers. Assume that $\Div_Y\psi$ is integral.
Let $(\mathcal{X},\mathcal{L})$ be the test configuration induced by $\psi$, namely
\[
\mathcal{X}=\Projrel_{\mathbb{C}} \bigoplus_{k\in \mathbb{Z}_{\geq 0}} \bigoplus_{j\in \mathbb{Z}_{\geq 0}} t^{-j} H^0(Y,kAL-j\Div_Y\psi)\,,
\]
where $A\in \mathbb{Z}_{>0}$ is a sufficiently divisible integer such that the algebra is generated in degree $k=1$. Take $\mathcal{L}=\mathcal{O}_{\mathcal{X}}(1)$.
Recall that the test curve induced by $(\mathcal{X},\mathcal{L})$ is $\psi^+_{\bullet}$. By taking further blowing-ups, we may assume that $\pi$ also resolves the singularities of all $\psi_{\tau}$, for $\tau\in \mathbb{Q}$, $\tau<\Psef(\psi)$.
The filtration induced by $(\mathcal{X},\mathcal{L})$ is
\[
\mathscr{F}^{\lambda}H^0(X,L^k)=H^0(Y,k\pi^*L-\lambda \Div_Y\psi)
\]
for $\lambda\geq 0$.

We slightly reformulate \cref{lma:volfiltc} in our setting.
\begin{lemma}\label{lma:eqvol}
For any $\tau<\Psef(\psi)$, we have
\begin{equation}\label{eq:intvol}
    \int_X \omega_{\psi^+_{\tau}}^n=\vol(\pi^*L-\tau\Div_Y\psi)\,.    
\end{equation}
\end{lemma}

\begin{corollary}\label{cor:dzd}
 For any $0\leq \tau< \Psef(\psi)$, the decomposition
 \[
\pi^*L-\tau \Div_Y\psi=(\pi^*L-\Div_Y\psi^+_{\tau})+N_{\tau}
 \]
 is the divisorial Zariski decomposition (\cite{Bou04}, \cite{Nak04}). More precisely, $(\pi^*L-\Div_Y\psi^+_{\tau})$ is the movable part of $\pi^*L-\tau \Div_Y\psi$.
\end{corollary}
\begin{proof}
Recall that by our definition of $\psi^+_{\tau}$, $N_{\tau}$ is effective. So the result follows from \cref{lma:eqvol} and \cite{FKL16}.
\end{proof}
\begin{remark}\label{rmk:gen}
Using \cref{lma:volfiltc}, one can easily generalize \cref{cor:dzd} to general test configurations.

We leave the details to the readers.
\end{remark}

In particular, 
\begin{equation}
D_L(\psi_{\tau}^+,\Redu\Div_Y\psi)=D_L(\psi_{\tau}^+,\Redu \Div_Y\psi^+_{\tau})\,.
\end{equation}
See \cref{subsec:diff} for the definition of $D_L$.

\begin{theorem}\label{thm:DFdreamy}
Let $(\mathcal{X},\mathcal{L})$ be as above, then
\[
\widetilde{\DF}(\mathcal{X},\mathcal{L})-\mathbf{E}_R(\psi^+)=\frac{1}{V}\int_{-\infty}^{\infty}D_L(\psi^+_{\tau},K_{Y/X})\,\mathrm{d}\tau+\frac{1}{V}\int_{-\infty}^{\infty} D_L(\psi^+_{\tau},\Div_Y\psi)\,\mathrm{d}\tau\,.
\]
When $\Div_Y\psi$ has a single irreducible component with coefficient $1$,
\[
\Ent^{\NA}(\ell^{\NA})=\mathbf{Ent}(\ell)=\Ent(\psi^+_{\bullet})\,.
\]
\end{theorem}
\begin{proof}
For the first part, the computation is a generalization of those in \cite{Fuj19}. As this method has been elaborated in  \cite[Section~3]{DL22}, we omit the proof.

As for the second part, assume that $\Div_Y\psi=F$ for some prime divisor $F$ over $X$. Then $\mathcal{X}_0$ is clearly irreducible and reduced and $\mathcal{X}$ is normal (see \cite[Lemma~3.2]{DL22} for example).  Hence we conclude by \cref{prop:BHJL}.
\end{proof}

\section{Variational approach on Berkovich spaces}\label{sec:vBer}
Let $X$ be a compact K\"ahler manifold of dimension $n$. Let $L$ be an ample line bundle on $X$. Fix a smooth strictly positively-curved metric $h$ on $L$. Let $\omega=c_1(L,h)$. 
In this section, we study the variation of the energy functional on the Berkovich space. As a consequence, we prove a comparison theorem of two entropy functionals \cref{thm:Entcomp2}, which is the key inequality used in comparing $\delta$-invariants.

\begin{proposition}\label{prop:NAdiff}
Let $\phi\in \mathcal{E}^{1,\NA}$. We have
\begin{equation}\label{eq:ENAD}
\varliminf_{\epsilon \to 0+} \frac{1}{\epsilon}\left(E^{\NA}(P[\phi+\epsilon A_X])-E^{\NA}(\phi) \right)\geq \frac{1}{V}\int_{X^{\An}} A_X\,\MA(\phi)\,. 
\end{equation}
Here $P[\bullet]$ is usc-regularized supremum of all  elements in $\mathcal{E}^{1,\NA}$ lying below $\bullet$. When $\phi\in \mathcal{H}^{\NA}$, the limit exists and equality holds.
\end{proposition}
\begin{remark}
    We expect that equality holds.
\end{remark}
\begin{proof}
Let $\mathcal{Y}$ run over the set of snc models of $X\times_{\mathbb{C}} \mathbb{C}((T))$ (see \cite[Section~1.3]{BJ18b} for the precise definition). Let $r_{\mathcal{Y}}:X^{\An}\rightarrow \Delta_{\mathcal{Y}}$ be the natural retraction. Let $f^{\mathcal{Y}}=A_X\circ r_{\mathcal{Y}}$. Then $(f^{\mathcal{Y}})_{\mathcal{Y}}$ is an increasing net of non-negative continuous functions converging to $A_X$ pointwisely. See \cite{JM12} for details.

Then for any $\epsilon>0$, $\phi+\epsilon f^{\mathcal{Y}}\leq \phi+\epsilon A_X$, so 
\[
E^{\NA}\left(P[\phi+\epsilon f^{\mathcal{Y}}]\right)\leq E^{\NA}\left(P[\phi+\epsilon A_X]\right)\,.
\]
Hence
\[
\varliminf_{\epsilon \to 0+} \frac{1}{\epsilon}\left(E^{\NA}(P[\phi+\epsilon A_X])-E^{\NA}(\phi) \right)\geq \frac{1}{V}\int_{X^{\An}} f^{\mathcal{Y}}\,\MA(\phi) 
\]
by \cite[Corollary~6.32]{BJ18b}.
By monotone convergence theorem (\cite[Proposition~7.12]{Fol99}), we conclude \eqref{eq:ENAD}.

Finally, let us deal with the case where $\phi$ is associated to some test configuration $(\mathcal{X},\mathcal{L})$. We may assume that $\mathcal{X}_0$ is snc. We claim that for any $\epsilon>0$,
\begin{equation}\label{eq:t12}
P[\phi+\epsilon f^{\mathcal{X}}]=P[\phi+\epsilon A_X]\,.
\end{equation}
We only have to prove
\begin{equation}\label{eq:t1}
P[\phi+\epsilon f^{\mathcal{X}}]\geq P[\phi+\epsilon A_X]\,.
\end{equation}
By \cite[Theorem~5.29]{BJ18b} (here we refer to the first version, this theorem does not appear in the final version),
\[
P[\phi+\epsilon A_X]\leq P[\phi+\epsilon A_X]\circ r_{\mathcal{X}}\,.
\]
But observe that
\[
P[\phi+\epsilon A_X]\circ r_{\mathcal{X}}\leq \phi+\epsilon f^{\mathcal{X}}\,.
\]
In fact, it suffices to check this on $\Delta_{\mathcal{X}}$, where the inequality follows from the definition of $P$.
Hence \eqref{eq:t1} follows and \emph{a fortiori} equality holds in \eqref{eq:ENAD}.
\end{proof}

For $\psi\in \PSH(X,\omega)$, each $\epsilon>0$, we define
\begin{equation}\label{eq:t4}
\psi^{\epsilon}:=\sups\left\{\,\varphi\in \PSH(X,\omega): \varphi\leq 0, \varphi^{\An}\leq \psi^{\An}+\epsilon A_X \text{ on } X^{\Div}_{\mathbb{Q}}\,\right\}\,.
\end{equation}
Note that $\psi^{\epsilon}$ is increasing and concave in $\epsilon>0$. By \cite[Theorem~6.1]{DDNL19log}, \cite{WN19}, $\log \int_X \omega_{\psi^{\epsilon}}^n$ is concave in $\epsilon$. When $\psi$ is $\mathscr{I}$-model, the mass $\int_X \omega_{\psi^{\epsilon}}^n$ is right-continuous at $\epsilon=0$.

\begin{lemma}\label{lma:derdiv1}
Let $\psi\in \PSH(X,\omega)$. Then for any birational model $\pi:Y\rightarrow X$,
\[
-\left.\frac{\mathrm{d}}{\mathrm{d}\epsilon}\right|_{\epsilon=0+}\Div_Y\psi^{\epsilon}\leq \sum_{E}A_X(E)E\leq\Redu \Div_Y\psi+K_{Y/X}\,,
\]
where $E$ runs over all irreducible divisors in $\Div_Y\psi$.
\end{lemma}
\begin{proof}
It follows from \eqref{eq:t4} that the only possible components of $-\left.\frac{\mathrm{d}}{\mathrm{d}\epsilon}\right|_{\epsilon=0+}\Div_Y\psi^{\epsilon}$ are components of $\Div_Y\psi$. Also by \eqref{eq:t4}, the multiplicity of each component $E$ is bounded by $A_X(E)$, hence
\[
-\left.\frac{\mathrm{d}}{\mathrm{d}\epsilon}\right|_{\epsilon=0+}\Div_Y\psi^{\epsilon}\leq \sum_{E}A_X(E)E\,.
\]
The second inequality is trivial.
\end{proof}
\begin{lemma}\label{lma:derdiv2}
Assume that $\psi\in \PSH(X,\omega)$ is $\mathscr{I}$-model and has positive mass. Then
\[
\frac{1}{V}\left.\frac{\mathrm{d}}{\mathrm{d}\epsilon}\right|_{\epsilon=0+}\int_X \omega_{\psi^{\epsilon}}^n\leq  \Ent([\psi])\,.
\]
\end{lemma}
\begin{proof}
By \cref{thm:volIm},
\begin{equation}\label{eq:t10}
\begin{aligned}
\left.\frac{\mathrm{d}}{\mathrm{d}\epsilon}\right|_{\epsilon=0+}\log\int_X \omega_{\psi^{\epsilon}}^n=&\left.\frac{\mathrm{d}}{\mathrm{d}\epsilon}\right|_{\epsilon=0+} \log \vol\left( L-\Div_{\mathfrak{X}}\psi^{\epsilon}\right)\\
\leq & \varliminf_{Y}\left.\frac{\mathrm{d}}{\mathrm{d}\epsilon}\right|_{\epsilon=0+} \log\vol\left( L-\Div_{Y}\psi^{\epsilon}\right)\\
=& \frac{n}{\int_X\omega_{\psi}^n}\varliminf_{Y} \left(\left\langle\pi^*L-\Div_Y\psi)^{n-1}\right\rangle\cdot \left(-\left.\frac{\mathrm{d}}{\mathrm{d}\epsilon}\right|_{\epsilon=0+}\Div_Y\psi^{\epsilon}\right)\right) \\
\leq & \frac{n}{\int_X\omega_{\psi}^n}\varliminf_{Y} \left(\left\langle\pi^*L-\Div_Y\psi)^{n-1}\right\rangle\cdot \left(\Redu \Div_Y\psi+K_{Y/X}\right)\right)\\
= & \frac{V}{\int_X\omega_{\psi}^n}\Ent([\psi])\,.
\end{aligned}
\end{equation}
Here $\pi:Y\rightarrow X$ runs over all birational models of $X$,
the second line follows from the log concavity of the masses of $\psi^{\epsilon}$ in $\epsilon$, the third line follows from \cite[Theorem~A]{BFJ09}, the fourth line follows from \cref{lma:derdiv1}.
\end{proof}

\begin{theorem}\label{thm:Entcomp2}
Let $\psi_{\bullet}\in \TC^1(X,\omega)$ be an $\mathscr{I}$-model test curve. Let $\ell$ be the geodesic ray defined by $\psi_{\bullet}$,
then
    \[
        \Ent^{\NA}(\ell^{\NA})\leq \Ent(\psi_{\bullet})\,.
    \]
\end{theorem}
\begin{proof}
We may assume that $\Ent(\psi_{\bullet})<\infty$. We first assume that $\psi_{\tau^+}$ has positive mass.

By Fatou's lemma, we always have
\[
\frac{1}{V}\bigintss_{-\infty}^{\infty} \left.\frac{\mathrm{d}}{\mathrm{d}\epsilon}\right|_{\epsilon=0+}\int_X \omega_{\psi^{\epsilon}_{\tau}}^n\,\mathrm{d}\tau\leq \varliminf_{\epsilon\to 0+}\epsilon^{-1}\bigintss_{-\infty}^{\infty} \left(\frac{1}{V}\int_X \omega_{\psi^{\epsilon}_{\tau}}^n-\frac{1}{V}\int_X \omega_{\psi_{\tau}}^n\right)\,\mathrm{d}\tau\,.
\]
On the other hand, since $\log\int_X\omega_{\psi_{\tau}^{\epsilon}}^n$ is concave in $\epsilon\geq 0$, 
fix $\epsilon_0>0$, for $\epsilon\in (0,\epsilon_0)$, we have
\[
V^{-1}\bigintss_{-\infty}^{\infty} \epsilon^{-1}\left(\int_X \omega_{\psi^{\epsilon}_{\tau}}^n-\int_X \omega_{\psi_{\tau}}^n\right)\,\mathrm{d}\tau\leq V^{-1}\bigintss_{-\infty}^{\tau^+} \epsilon^{-1} \left(\int_X\omega_{\psi_{\tau}^{\epsilon_0}}^n\right)\left(\log \int_X \omega_{\psi^{\epsilon}_{\tau}}^n-\log \int_X \omega_{\psi_{\tau}}^n \right)\,\mathrm{d}\tau\,.
\]
Thus by monotone convergence theorem,
\[
\varlimsup_{\epsilon\to 0+}\epsilon^{-1}\bigintss_{-\infty}^{\infty} \left(\frac{1}{V}\int_X \omega_{\psi^{\epsilon}_{\tau}}^n-\frac{1}{V}\int_X \omega_{\psi_{\tau}}^n\right)\,\mathrm{d}\tau\leq 
\frac{1}{V}\bigintss_{-\infty}^{\tau^+}\left(\int_X\omega_{\psi_{\tau}^{\epsilon_0}}^n\right)\left( \int_X\omega_{\psi_{\tau}}^n\right)^{-1} \left.\frac{\mathrm{d}}{\mathrm{d}\epsilon}\right|_{\epsilon=0+}\int_X \omega_{\psi^{\epsilon}_{\tau}}^n\,\mathrm{d}\tau\,.
\]
Let $\epsilon_0\to 0+$, by dominated convergence theorem, we get
\[
\varlimsup_{\epsilon\to 0+}\epsilon^{-1}\bigintss_{-\infty}^{\infty} \left(\frac{1}{V}\int_X \omega_{\psi^{\epsilon}_{\tau}}^n-\frac{1}{V}\int_X \omega_{\psi_{\tau}}^n\right)\,\mathrm{d}\tau\leq \frac{1}{V}\bigintss_{-\infty}^{\infty} \left.\frac{\mathrm{d}}{\mathrm{d}\epsilon}\right|_{\epsilon=0+}\int_X \omega_{\psi^{\epsilon}_{\tau}}^n\,\mathrm{d}\tau\,.
\]
Thus
\[
\begin{split}
\left.\frac{\mathrm{d}}{\mathrm{d}\epsilon}\right|_{\epsilon=0+}E^{\NA}(\psi^{\epsilon}_{\bullet})
=&\left.\frac{\mathrm{d}}{\mathrm{d}\epsilon}\right|_{\epsilon=0+}\bigintss_{-\infty}^{\infty} \left(\frac{1}{V}\int_X \omega_{\psi^{\epsilon}_{\tau}}^n-1\right)\,\mathrm{d}\tau\\
=&\frac{1}{V}\bigintss_{-\infty}^{\infty} \left.\frac{\mathrm{d}}{\mathrm{d}\epsilon}\right|_{\epsilon=0+}\int_X \omega_{\psi^{\epsilon}_{\tau}}^n\,\mathrm{d}\tau\\
\leq & \int_{-\infty}^{\infty}\Ent([\psi_{\tau}])\,\mathrm{d}\tau\,,
\end{split}
\]
where the first equality follows from \cref{thm:LegE}.
By \cref{thm:legna}, the non-Archimedean potential associated to $\psi^{\epsilon}_{\bullet}$ is just $P[\ell^{\NA}+\epsilon A_X]$, hence we can apply \cref{prop:NAdiff} to conclude.

For a general $\psi_{\bullet}$, for each $\delta>0$, we define a new test curve $\psi^{\tau}$ that agrees with $\psi_{\tau}$ when $\tau\leq \tau^+-\delta$ and equals $-\infty$ otherwise. We apply the previous step and the fact that $\Ent^{\NA}(\bullet)$ is lsc. 
\end{proof}

The same proof actually yields equality in the case of test configurations.
\begin{corollary}\label{cor:eqholds}
Let $(\mathcal{X},\mathcal{L})$ be a test configuration of $(X,L)$. Let $\ell$ be the induced Phong--Sturm geodesic ray, let $\psi=\hat{\ell}$. Then
    \[
        \Ent^{\NA}(\ell^{\NA})=\mathbf{Ent}(\ell)= \Ent(\psi_{\bullet})\,.
    \]
\end{corollary}
\begin{proof}
The first equality follows from \cref{prop:BHJL}. The inequality $\Ent^{\NA}(\ell^{\NA})\leq \Ent(\psi_{\bullet})$ follows from \cref{thm:Entcomp2}. 

Now we prove the converse. Replacing $\mathcal{L}$ by $(1+\delta)\mathcal{L}$ for a small $\delta\in \mathbb{Q}_{>0}$, we may assume that $\psi_{\tau^+}$ has positive mass.
We may assume that $\mathcal{X}_0=\sum b_E E$ is snc and $\mathcal{X}$ dominates $X\times \mathbb{C}$ by a map $\Pi:\mathcal{X}\rightarrow X\times \mathbb{C}$. Let $D$ be a divisor supported on the central fibre and $\mathcal{O}(D)=\mathcal{L}-p_1^*L$, where $p_1:X\times \mathbb{C} \rightarrow X$ is the natural map. 

Observe that $\psi^{\epsilon}_{\bullet}$ is the test curve defined by the (not necessarily finitely generated) $\mathbb{Z}$-filtration $\mathscr{F}_{\epsilon}$ associated to the model $(\mathcal{X},\mathcal{L}+\epsilon K^{\log}_{\mathcal{X}/X\times \mathbb{C}})$. 
In fact, this follows from \cite[Corollary~4.12]{BHJ17}, \cite[Theorem~8.5]{BFJ16} and \eqref{eq:t12} (see also discussions in \cite{Li20} after Definition~2.7).
By \cite[Corollary~4.12]{BHJ17}, \cite[Proof of Lemma~5.17]{BHJ17}, $\mathscr{F}_{\epsilon}$ is given by
\begin{equation}\label{eq:t11}
\mathscr{F}_{\epsilon}^{\lambda}H^0(X,L^k)=\left\{\,s\in H^0(X,L^k): r(\ord_E)(s)+k\ord_E D+k\epsilon A_X(r(\ord_E)) \geq b_E\lambda\,,\forall E \,\right\}\,,
\end{equation}
where $E$ runs over all components of $\mathcal{X}_0$, $r(\ord_E)$ is the restriction of $\ord_E$ to $\mathbb{C}(X)$. Recall that $r(\ord_E)$ is a divisorial valuation (\cite[Section~4.2]{BHJ17}).  Let $\pi:Y\rightarrow X$ be a birational model on which the divisors corresponding to all $r(\ord_E)$ lie and which resolves the singularities of all $\psi_{\tau}$, which is possible by \cref{cor:tcana}. Then $\pi^*L-\Div_Y\psi_{\tau}$ is nef by \cref{lma:diffnab}.
Now by \cref{lma:volfiltc} and \eqref{eq:t11},
\[
\begin{split}
\left.\frac{\mathrm{d}}{\mathrm{d}\epsilon}\right|_{\epsilon=0+}\int_X \omega_{\psi^{\epsilon}_{\tau}}^n
\geq & \left.\frac{\mathrm{d}}{\mathrm{d}\epsilon}\right|_{\epsilon=0+}\vol\left(\pi^*L- \Div_Y\psi_{\tau}+\epsilon\sum_F A_X(F)F \right)\\
=&n\left(\pi^*L- \Div_Y\psi_{\tau}\right)^{n-1}\cdot \sum_F A_X(F)F\\
=&n\left(\pi^*L- \Div_Y\psi_{\tau}\right)^{n-1}\cdot \left(K_{Y/X}+\Redu\Div_Y \psi_{\tau} \right)\,,
\end{split}
\]
where $F$ runs over all irreducible components of $\Div_Y\psi_{\tau}$, the second step follows from \cref{thm:diffv}. In the first and the last step, we applied the negativity lemma (\cite[Lemma~3.39]{KM08}).
We conclude by the same argument as above.
\end{proof}

\section{Stability thresholds}\label{sec:delta}
Let $X$ be a compact Kähler manifold of dimension $n$. Let $L$ be an ample line bundle on $X$. Fix a smooth strictly positively-curved Hermitian metric $h$ on $X$ and let $\omega=c_1(L,h)$.

We will compare various $\delta$ invariants and prove the main theorem of the paper \cref{thm:deltacomp}.

We refer to \cref{def:delta2} and \cref{def:delta} for the definitions of $\delta_{\mathrm{pp}}$ and $\delta'$.

\begin{proposition}\label{prop:deltaineq}
We always have $\delta'\geq \delta$.
\end{proposition}
\begin{proof}
Let $\psi\in \PSH(X,\omega)$ be an unbounded potential with analytic singularities. 
Let $\ell$ be the geodesic ray induced by the generalized deformation to the normal cone with respect to $\psi$ (see \cref{subsec:defn}).
By \cref{thm:Entcomp2} and \eqref{eq:trivialint},
\begin{equation}\label{eq:ineqentdefnc}
    \Ent^{\NA}(\MA(\ell^{\NA}))\leq \Ent(\psi_{\bullet})=\frac{1}{V}\left(K_{Y/X}\cdot (-\Div_Y\psi)^{n-1}\right)+\frac{n}{V}\left(G_{n-1}(L,\Div_Y\psi)\cdot \Redu \Div_Y\psi\right)\,.
\end{equation}
By \cref{cor:Jtildeslope}, we have
\[
\tilde{\mathbf{J}}(\ell)=\frac{n}{V}\int_0^1 \left(\int_X \omega\wedge \omega_{\tau\psi}^{n-1}-\int_X  \omega_{\tau\psi}^{n}\right)\,\mathrm{d}\tau\,.
\]
By our definition (see also \cite[Proposition~2.38]{Li20}),
\[
\tilde{J}^{\NA}(\ell^{\NA})=\tilde{\mathbf{J}}(\ell)\,.
\]
Hence by \eqref{eq:de1} and \eqref{eq:EstJt},
\[
\delta\leq  \frac{\Ent^{\NA}(\MA(\ell^{\NA}))}{E^*(\MA(\ell^{\NA}))}\leq  \frac{(K_{Y/X}\cdot (-\Div_Y\psi)^{n-1})+n\left(G_{n-1}(L,\Div_Y\psi)\cdot \Redu \Div_Y\psi\right)}{n\int_0^1 \left(\int_X \omega\wedge \omega_{\tau\psi}^{n-1}-\int_X  \omega_{\tau\psi}^{n}\right)\,\mathrm{d}\tau}\,.
\]
\end{proof}

Similarly, we have
\begin{proposition}\label{prop:deltaineq1}
We always have $\delta_{\mathrm{pp}}\geq \delta$.
\end{proposition}
\begin{proof}
Let $\psi\in \PSH(X,\omega)$ be an $\mathscr{I}$-model potential. Let $\ell$ be the geodesic ray induced by $\psi^+_{\bullet}$. Then by \cref{thm:Entcomp2},
\[
\Ent^{\NA}(\MA(\ell^{\NA}))\leq  \Ent(\psi^+_{\bullet})\,.
\]
While $\tilde{J}^{\NA}(\ell)=\tilde{\mathbf{J}}(\psi^+_{\bullet})$ as in the previous proof. Hence
\[
\delta\leq  \frac{\Ent^{\NA}(\MA(\ell^{\NA}))}{E^*(\MA(\ell^{\NA}))}\leq \frac{\Ent(\psi^+_{\bullet})}{\tilde{\mathbf{J}}(\psi^+_{\bullet})}\,.
\]
We conclude (c.f. \cref{rmk:deltappImdoeldep}).
\end{proof}

\begin{theorem}\label{thm:deltacomp}
Assume that $X$ is Fano, $L=-K_X$. If $\delta<\frac{n+1}{n}$, then $\delta\geq \delta_{\mathrm{pp}}$. 
\end{theorem}
\cref{prop:deltaineq1} and \cref{thm:deltacomp} together are just \cref{thm:introdelta1} in the introduction.

\begin{proof}
Assume that $\delta< 1$, by \cite[Proof of Theorem~4.1]{BLZ19}, $\delta$ can be computed by a sequence of extractable divisors, say $E_k$ in the sense that
\[
\delta=\lim_{k\to\infty}\frac{A_X(E_k)}{S_L(E_k)}\,.
\]
Let $\pi_k:(Y_k,\Delta_k)\rightarrow X$ be a dlt (divisorially log terminal) extractions of $E_k$ so that $E_k$ is the only exceptional divisor and $A_X(E_k)\in (0,1)$.
See \cite[Corollary~1.38]{Kol13} for the notion and existence of dlt extraction.
Choose $\epsilon'\in \mathbb{Q}_{>0}$, so that $\pi^*L-\epsilon' E_k$ is semi-ample. Take $m\in \mathbb{Z}_{>0}$, so that $m(\pi^*L-\epsilon' E_k)$ is base-point free. 
Take a basis $s_1,\ldots,s_M$ of $H^0(Y_k,m(\pi^*L-\epsilon' E_k))$, regarded as a subspace of $H^0(X,L^m)$. Let
\[
\psi_k:=\frac{1}{m}\log \max_{i=1,\ldots,M} |s_i|^2_{h^m}\,.
\]
The filtration induced by $\psi_k$ on $R(X,L^m)$ in the sense of \cref{ex:ext} is the same as that defined by $\ord_{E_k}$.
So the geodesic ray $\ell_k$ induced by $\psi_k$ through the extended deformation to the normal cone construction is the same as the geodesic ray induced by the filtration of $\ord_{E_k}$. By \cref{thm:DFdreamy} and \cref{cor:Jtildeslope}, 
\[
\frac{A_X(E_k)}{S_L(E_k)}\geq \delta_{\mathrm{pp}}\,.
\]
Let $k\to\infty$, we conclude.

Now assume that $\delta=1$. By \cite[Theorem~6.7]{BLZ19}, there is a prime divisor $E$ over $X$ computing $\delta(X)$. By \cite[Proof of Theorem~4.5]{BLZ19}, $E$ is extractable. So we can proceed as in the case $\delta<1$.

In general, if $\delta<\frac{n+1}{n}$, it suffices to apply \cite{LXZ21} instead of \cite{BLZ19} and run the same arguments.
\end{proof}

\begin{remark}
By slightly refining the argument, one finds that when $\delta<\frac{n+1}{n}$, there is always a qpsh function with analytic singularities that computes $\delta_{\mathrm{pp}}$.
\end{remark}

\begin{corollary}
Assume that $X$ is Fano and $L=-K_X$. Then $\delta_{\mathrm{pp}}\geq 1$ (resp. $\delta_{\mathrm{pp}}> 1$) if{f} $X$ is K-semistable (resp. uniformly K-stable).
\end{corollary}

\section{Further problems}\label{sec:pro}
\subsection{Minimizers}\label{subsec:min}
Let $X$ be a Fano manifold and $L=-K_X$. We assume that $\delta<1$. Fix a smooth strictly positively-curved Hermitian metric $h$ on $L$. Let $\omega=c_1(L,h)$.
In this case, it is well-known that $\delta$ is equal to the greatest Ricci lower bound $R(X)$:
\[
R(X):=\sup\left\{\,t\in [0,1]: \exists \omega\in c_1(X) \text{ s.t. } \Ric \omega>t\omega\,\right\}\,.
\]
This quantity was first explicitly introduced by Rubinstein in \cite{Rub08}, \cite{Rub09}.
See \cite{Sz11} for further results. This invariant also appears in an implicit form in \cite{Tian92}.
One could always solve Aubin's continuity path (\cite{Aub84}) for $t<R(X)$:
\[
\omega_{\varphi_t}^n=e^{F-t\varphi_t}\omega^n\,,
\]
where $F$ is the Ricci potential of $\omega$: $\Ric\omega-\omega=\ddc F$, $\int_X (\exp(F)-1)\omega^n=0$.

The following are known about $\varphi_t$:
\begin{enumerate}
    \item Blowing-up at the limit time:
    \[
    \lim_{t\to R(X)-} \sup_X \varphi_t=\infty\,.
    \]
    See \cite{Siu88}, \cite{Tian87}.
    \item There is a proper closed subvariety $V\subseteq X$, such that on each compact subset of $X\setminus V$, for any increasing sequence $t_i\to R(X)$ and $t_i<R(X)$, up to passing to a subsequence, $\omega_{\varphi_{t_i}}^n$ converges to $0$ uniformly (\cite{Tos12}).
    \item Tian's partial $C^0$-estimate: let $\beta_{m,t}$ be the $m$-th Bergman kernel defined by $\omega_{\varphi_t}$. Then there exists $m\in \mathbb{Z}_{>0}$ and $C>0$, such that
    \[
    \inf_X \rho_{m,t}\geq C^{-1}
    \]
    for any $t\in [0,R(X))$. See \cite{Sz16}, \cite{LS20},\cite{Zhangc0}, \cite{CW20}, \cite{Bam18}.
\end{enumerate}
It follows from the partial $C^0$-estimate that for any increasing sequence $t_i\to R(X)$, $t_i<R(X)$, up to subtracting a subsequence, there is $G\in \PSH(X,\omega)$, such that
\[
\varphi_{t_i}-\sup_X \varphi_{t_i}\to G
\]
in $L^1$. Moreover, $G$ has the following type of singularities:
\[
\frac{1}{m}\log \sum_{j=1}^N \lambda_j^2 |S_j|_{h^m}^2\,,
\]
where $m\in \mathbb{Z}_{>0}$, $\lambda_j\in (0,1]$, $S_j\in H^0(X,K_X^{-m})$.
Moreover, $\omega_{G}^n=0$.
See \cite{McCT19} for details. The function $G$ is known as a \emph{pluricomplex Green function} of $X$.

We make the following conjecture:
\begin{conjecture}
When $\delta<1$, the pluricomplex Green function $G$ is a minimizer of $\delta_{\mathrm{pp}}$. 
\end{conjecture}

\subsection{Moser--Trudinger type inequalities}
Let $X$ be a compact Kähler manifold of dimension $n$. Let $[\omega]$ be a Kähler class on $X$ with a representative Kähler form $\omega$.

\begin{definition}[{\cite[Definition~3.1, Proposition~3.5]{Zhangconti}}]
 We define the \emph{analytic $\delta$-invariant} $\delta^A$ of $[\omega]$ as
 \begin{equation}\label{eq:MT}
\begin{split}
\delta^A([\omega]):=&\sup\left\{\,\lambda>0: \int_X e^{-\lambda\left( \varphi-E(\varphi)\right)}=\mathcal{O}_{\lambda}(1)\text{ for any }\varphi\in \mathcal{H}(X,\omega)\,\right\}\\
=&\sup\left\{\,\lambda>0: \Ent(\varphi)\geq  \lambda \tilde{J}(\varphi)-\mathcal{O}_{\lambda}(1)\text{ for any }\varphi\in \mathcal{H}(X,\omega)\,\right\}\,.
\end{split}
\end{equation}
\end{definition}
Inequalities as in the first line of \eqref{eq:MT} are known as \emph{Moser--Trudinger type inequalities}, they were first studied in \cite{BB11}. See \cite{DNGL21} for recent progress in Moser--Trudinger type inequalities. The equality of two lines in \eqref{eq:MT} follows essentially from \cite[Proposition~4.11]{BBEGZ16}, as explained in \cite[Proposition~3.5]{Zhangconti}.

In \cite{Zhang21}, Zhang proved that $\delta^A([\omega])=\delta(L)$, improving previous partial results in \cite[Proposition~3.11]{Zhangconti}, \cite[Proposition~5.3]{RTZ21}. Hence in this case, $\delta^A\leq \delta_{\mathrm{pp}}$ by \cref{prop:deltaineq1}.
Moreover, both $\delta^A$ and our $\delta_{\mathrm{pp}}$ make sense for a transcendental Kähler class. It is interesting to understand the exact relation between $\delta^A$ and $\delta_{\mathrm{pp}}$.

\subsection{Non-Archimedean entropy in terms of test curves}
Let $X$ be a compact K\"ahler manifold of dimension $n$. Let $\omega$ be a K\"ahler form on $X$.

When $\omega$ is in the first Chern class of an ample $\mathbb{Q}$-line bundle, $\mathcal{E}^{1,\NA}(L)$ makes sense as in \cite{BJ18b}. In general, we define $\mathcal{E}^{1,\NA}([\omega])$ as the subspace of $\mathcal{R}^1$ consisting of $\ell\in \mathcal{R}^1$, such that $\hat{\ell}_{\tau}$ is either $-\infty$ or $\mathscr{I}$-model for all $\tau$. 

\begin{conjecture}
Let $\ell\in \mathcal{R}^1$. Assume that $\mathbf{Ent}(\ell)<\infty$, then $\ell\in \mathcal{E}^{1,\NA}(L)$.
\end{conjecture}
When $[\omega]$ is integral, this follows from \cite{Li20}.

\begin{conjecture}\label{conj:ent}
Let $\ell\in \mathcal{E}^{1,\NA}$, let $\psi=\hat{\ell}$, then
\[
\Ent^{\NA}(\ell^{\NA})=\mathbf{Ent}(\ell)=\Ent(\psi_{\bullet})\,.
\]
\end{conjecture}
Many special cases are known: when $[\omega]$ is integral, we know that $\Ent^{\NA}(\ell^{\NA})\leq \mathbf{Ent}(\ell)$ (\cref{prop:BHJL}), $\Ent^{\NA}(\ell^{\NA})\leq \Ent(\psi_{\bullet})$ (\cref{thm:Entcomp2}). When $\ell$ is the Phong--Sturm geodesic ray of some test configuration, both equalities hold (\cref{cor:eqholds}). 

When $[\omega]$ is not integral, all three terms are still defined, but very little information is known.

We summarize the information we know so far about various functionals in \cref{tbl:Comp}.
\begin{table}[h]\caption{Comparison of functionals}\label{tbl:Comp}
\begin{center}
\begin{tabular}{ |c | c | c | c| } 
 \hline
 Maximal geodesic rays & NA potentials & Test curves  & Known facts \\ 
 \hline 
 $\mathbf{E}$ & $E^{\NA}$ & $\mathbf{E}$ & All equal \\
 \hline
 $\mathbf{E}_R$ & $E_R^{\NA}$ & $\mathbf{E}_R$ & All equal \\ 
 \hline
 $\mathcal{L}^{\NA}_k$ & ? & $\mathcal{L}^{\NA}_k$ & First=Third\\
 \hline
 $\mathbf{Ent}$ & $\Ent^{\NA}$ & $\Ent$ & \begin{tabular}{@{}c@{}} Second $\leq$ First\\ Second$\leq$ Third\end{tabular} \\
 \hline
\end{tabular}
\end{center}
\end{table}

This missing term in \cref{tbl:Comp} is given by a construction similar to the relative volume in \cite[(0.1)]{BE21} up to an error term. One evidence of this is given by the analogy between \cite[Theorem~1.1]{BGM21} and \cite[Theorem~1.2]{DX22}.
Note that every term on the third column is defined as an integral of some functional of psh singularities along the test curve.

Finally, let us explain the relation between \cref{conj:ent} and the celebrated Yau--Tian--Donaldson (YTD) conjecture. Up to now, it is well-understood that in order to achieve the variational approach of the YTD conjecture, it suffices to show that for a maximal geodesic ray, $\mathbf{Ent}(\ell)$ is continuous along the approximation of Berman--Boucksom--Jonsson  (\cite{BBJ15}, \cite{Li20}, \cite{CC1}, \cite{CC2}, \cite{CC3}). If \cref{conj:ent} holds, up to some technical subtleties, the problem can be reduced to showing that $\Ent([\bullet])$ of a psh singularity is continuous along a suitable quasi-equisingular approximation.

\newpage

\newpage
\printbibliography

\bigskip
  \footnotesize

  Mingchen Xia, \textsc{Department of Mathematics, Chalmers Tekniska Högskola, G\"oteborg}\par\nopagebreak
  \textit{Email address}, \texttt{xiam@chalmers.se}\par\nopagebreak
  \textit{Homepage}, \url{http://www.math.chalmers.se/~xiam/}.
\end{document}